\documentclass{amsart}%
\usepackage{amssymb}
\usepackage{amsfonts}
\usepackage{geometry}
\usepackage{graphicx}
\usepackage{amsmath}%
\providecommand{\U}[1]{\protect\rule{.1in}{.1in}}
\newtheorem{theorem}{Theorem}
\theoremstyle{plain}

\newtheorem{conjecture}[theorem]{Conjecture}

\newtheorem{definition}[theorem]{Definition}

\newtheorem{lemma}[theorem]{Lemma}
\newtheorem{notation}[theorem]{Notation}

\newtheorem{remark}[theorem]{Remark}

\numberwithin{equation}{section}

\begin{document}
\title[Trilinear characterizations of the Fourier extension conjecture]{Trilinear characterizations of the Fourier extension conjecture on the
paraboloid in three dimensions}
\author{Cristian Rios}
\address{University of Calgary\\
Calgary, Alberta, Canada}
\email{crios@ucalgary.ca}
\author{Eric Sawyer}
\address{McMaster University\\
Hamilton, Ontario, Canada}
\email{sawyer@mcmaster.ca}
\thanks{Eric Sawyer's research supported in part by a grant from the National Sciences
and Engineering Research Council of Canada}
\maketitle

\begin{abstract}
We first prove that a local trilinear extension inequality on the paraboloid
in $\mathbb{R}^{3}$ is equivalent to the Fourier restriction conjecture in
$\mathbb{R}^{3}$, which may be viewed as a companion result to either the
transversal trilinear inequality of Bennett, Carbery and Tao in \cite{BeCaTa},
or to the bilinear characterization due to Tao, Vargas and Vega in
\cite{TaVaVe}. Namely, we prove that the paraboloid restriction conjecture
holds in $\mathbb{R}^{3}$ \emph{if and only if} for every $q>3$ there is
$\nu>0$ such that the trilinear inequality,
\[
\left\Vert \mathcal{E}f_{1}\mathcal{E}f_{2}\mathcal{E}f_{3}\right\Vert
_{L^{\frac{q}{3}}\left(  B\left(  0,R\right)  \right)  }\leq C_{\nu
}R^{\varepsilon}\mathop{\displaystyle \prod }\limits_{k=1}^{3}\left\Vert
f_{k}\right\Vert _{L^{\infty}\left(  U\right)  }\ ,\ \ \ \ \ \text{for
}R>1,\varepsilon>0
\]
holds when taken over all $f_{k}\in L^{\infty}\left(  U_{k}\right)  $, and all
triples $\left(  U_{1},U_{2},U_{3}\right)  $ of squares that satisfy the $\nu
$-disjoint condition%
\[
\mathop{\rm diam}\left[  \Phi\left(  U_{k}\right)  \right]  \approx\nu\text{
and }\mathop{\rm dist}\left[  \Phi\left(  U_{k}\right)  ,\Phi\left(
U_{j}\right)  \right]  \geq\nu,\text{ for }1\leq j,k\leq3,
\]
where $\Phi$ is the standard parameterization of the paraboloid $\mathbb{P}%
^{2}$ and $\mathcal{E}f=\left[  \Phi_{\ast}\left(  f\left(  x\right)
dx\right)  \right]  ^{\wedge}$ is the associated Fourier extension operator.
The proof follows an argument of Bourgain and Guth \cite{BoGu}, but exploiting
the fact that the problematic Case 3 in their argument dissolves due to the
$\nu$-disjoint assumption, as opposed to the $\nu$-transversal assumption in
\cite{BeCaTa}.

Then we prove the equivalence of the above trilinear Fourier extension
conjecture with the special case of testing a local trilinear inequality over
certain smooth Alpert pseudoprojections $\mathsf{Q}_{s,U}^{\eta}f$,
representing the weakest such inequality equivalent to the Fourier extension
conjecture that the authors could find. This special inequality has the
following form. For every $q>3$ there is $\nu>0$ such that
\begin{align*}
&  \left(  \int_{B\left(  0,2^{s}\right)  \setminus B\left(  0,2^{s-1}\right)
}\left(  \left\vert \mathcal{E}\mathsf{Q}_{s_{1},U_{1}}^{\eta}f_{1}\left(
\xi\right)  \right\vert \ \left\vert \mathcal{E}\mathsf{Q}_{s_{2},U_{2}}%
^{\eta}f_{2}\left(  \xi\right)  \right\vert \ \left\vert \mathcal{E}%
\mathsf{Q}_{s_{3},U_{3}}^{\eta}f_{3}\left(  \xi\right)  \right\vert \right)
^{\frac{q}{3}}\ d\xi\right)  ^{\frac{3}{q}}\\
&  \ \ \ \ \ \ \ \ \ \ \ \ \ \ \ \ \ \ \ \ \lesssim2^{\varepsilon s}\left\Vert
f_{1}\right\Vert _{L^{\infty}}\left\Vert f_{2}\right\Vert _{L^{\infty}%
}\left\Vert f_{3}\right\Vert _{L^{\infty}}\ ,
\end{align*}
for all $\nu$-disjoint triples $\left(  U_{1},U_{2},U_{3}\right)  $, all
$\varepsilon>0$, $s\in\mathbb{N}$ and $s_{1}\leq s_{2}\leq s_{3}$ with
$s_{2},s_{3}$ arbitrarily close to $s$.

The extension to other quadratic surfaces of positive Gaussian curvature in
$\mathbb{R}^{3}$ is straightforward.

\end{abstract}
\tableofcontents

\section{Introduction}

In this paper we show that a variant of the trilinear Fourier extension
inequality of Bennett, Carbery and Tao \cite{BeCaTa}\footnote{See \cite{Tao2}
for a sharpening of these results.} is equivalent to the Fourier extension
conjecture in dimension three, at least for quadratic surfaces of positive
Gaussian curvature, which include the paraboloid. This can also be viewed as a
companion result to the bilinear characterizations in Tao, Vargas and Vega
\cite{TaVaVe}. Of course the trilinear characterization implies the
corresponding bilinear one by H\"{o}lder's inequality\footnote{applied with
exponents $\left(  3,3,3\right)  $ to the factorization%
\[
\left\vert \mathcal{E}f_{1}\mathcal{E}f_{2}\mathcal{E}f_{3}\right\vert
=\sqrt{\left\vert \mathcal{E}f_{1}\mathcal{E}f_{2}\right\vert }\sqrt
{\left\vert \mathcal{E}f_{2}\mathcal{E}f_{3}\right\vert }\sqrt{\left\vert
\mathcal{E}f_{3}\mathcal{E}f_{1}\right\vert }.
\]
It is possible that the bilinear proof in \cite{TaVaVe} can be adapted to the
disjoint trilinear setting, but the Whitney condition in three dimensions
complicates matters, and we will not pursue this here.}.

We will weaken the transversality hypothesis in \cite{BeCaTa} to a disjoint
hypothesis (while retaining the embeddings into the paraboloid), as well as
weakening the bound in the conclusion, and then adapt an argument of Bourgain
and Guth \cite{BoGu} to derive the Fourier extension conjecture as a
consequence of this disjoint Fourier trilinear inequality. The key point here
is that the problematic Case 3 in the argument of \cite[Section 2]{BoGu} is
eliminated, along with their restriction to $p>\frac{10}{3}$, by the disjoint assumption.

In Theorem \ref{main Alpert} in the third section of the paper, we state and
prove another characterization of the Fourier extension inequality, in which
we restrict the functions in the disjoint trilinear conjecture to special
smooth Alpert pseudoprojections, representing the `simplest' characterization
that the authors could find. In the appendix immediately following that, we
sketch how a square function modification of the arguments used here could
give an alternate, and arguably simpler, proof of the \emph{probabilistic}
Fourier extension theorem in \cite{Saw7} for the paraboloid in three
dimensions\footnote{The result in \cite{Saw7} was stated only for the sphere,
but easily extends to smooth compact surfaces with positive Hessian.}.

We will only give details of proofs for our results in the special case of the
paraboloid $\mathbb{P}^{2}$ in three dimensions here, since the arguments are
virtually the same for quadratic surfaces with positive Hessian in
$\mathbb{R}^{3}$. Denote the Fourier extension operator $\mathcal{E}$\ by,%
\[
\mathcal{E}f\left(  \xi\right)  \equiv\left[  \Phi_{\ast}\left(  f\left(
x\right)  dx\right)  \right]  ^{\wedge}\left(  \xi\right)
,\ \ \ \ \ \text{for }\xi\in\mathbb{R}^{3},
\]
where $\Phi_{\ast}\left(  f\left(  x\right)  dx\right)  $ denotes the
pushforward of the measure $f\left(  x\right)  dx$ supported in $U\subset
B_{\mathbb{R}^{2}}\left(  0,\frac{1}{2}\right)  $ to the paraboloid
$\mathbb{P}^{2}$ under the usual parameterization $\Phi:U\rightarrow
\mathbb{P}^{2}$ by $\Phi\left(  x\right)  =\left(  x_{1},x_{2},x_{1}^{2}%
+x_{2}^{2}\right)  $ for $x=\left(  x_{1},x_{2}\right)  \in U$.

\begin{conjecture}
[Fourier extension]The Fourier extension conjecture for the paraboloid
$\mathbb{P}^{2}$ in $\mathbb{R}^{3}$ is the assertion that%
\begin{equation}
\left\Vert \mathcal{E}f\right\Vert _{L^{q}\left(  \mathbb{R}^{3}\right)
}\lesssim\left\Vert f\right\Vert _{L^{q}\left(  U\right)  }%
,\ \ \ \ \ \text{for }q>3.\label{FEC}%
\end{equation}
which we also denote by $\mathcal{E}\left(  \otimes_{1}L^{q}\rightarrow
L^{q}\right)  $.
\end{conjecture}

\subsection{Statements of main theorems}

The analogous Fourier extension conjecture for the $n$-dimensional sphere
$\mathbb{S}^{n-1}$ was made in 1967 by E. M. Stein, see e.g. \cite[see the
Notes at the end of Chapter IX, p. 432, where Stein proved the restriction
conjecture for $1\leq p<\frac{4n}{3n+1}$]{Ste2} and \cite{Ste}. The
two-dimensional case of the Fourier extension conjecture was proved over half
a century ago by L. Carleson and P. Sj\"{o}lin \cite{CaSj}, see also C.
Fefferman \cite{Fef} and A. Zygmund \cite{Zyg}. A web search reveals much
progress on extension theorems, as well as Kakeya theorems, in the ensuing
years. In particular, the Kakeya set conjecture has recently been proved in
dimension $n=3$ by Hong Wang and Joshua Zahl \cite{WaZa}.

Here is our weakening of the transversality condition as introduced by Muscalu
and Oliveira \cite{MuOl}, which we refer to as a \emph{disjoint} condition.

\begin{definition}
Let $\varepsilon,\nu>0$, $1<q<\infty$ and $q\leq p\leq\infty$. Denote by
$\mathcal{E}_{\mathop{\rm disj}\nu}\left(  \otimes_{3}L^{p}\rightarrow
L^{\frac{q}{3}};\varepsilon\right)  $ the trilinear Fourier extension
inequality%
\begin{equation}
\left\Vert \mathcal{E}f_{1}\ \mathcal{E}f_{2}\ \mathcal{E}f_{3}\right\Vert
_{L^{\frac{q}{3}}\left(  B\left(  0,R\right)  \right)  }\leq\left(
C_{\varepsilon,\nu,p,q}R^{\varepsilon}\right)  ^{3}\left\Vert f_{1}\right\Vert
_{L^{p}\left(  U_{1}\right)  }\left\Vert f_{2}\right\Vert _{L^{p}\left(
U_{2}\right)  }\left\Vert f_{3}\right\Vert _{L^{p}\left(  U_{3}\right)
}\ ,\label{concl thm}%
\end{equation}
taken over all $R\geq1$, $f_{k}\in L^{p}\left(  U_{k}\right)  $, and all
triples $\left(  U_{1},U_{2},U_{3}\right)  \subset U^{3}$ that satisfy the
$\nu$-disjoint condition,%
\begin{equation}
\mathop{\rm diam}\left[  \Phi\left(  U_{k}\right)  \right]  \approx\nu\text{
and }\mathop{\rm dist}\left[  \Phi\left(  U_{k}\right)  ,\Phi\left(
U_{j}\right)  \right]  \geq\nu,\text{ for }1\leq j,k\leq3,\label{weak sep}%
\end{equation}
and where the constant $C_{\varepsilon,\nu,p,q}$ is independent of $R\geq1 $
and the functions $f_{k}\in L^{p}\left(  U_{k}\right)  $.
\end{definition}

Note that (\ref{concl thm}) is invariant under translation of the ball
$B\left(  0,R\right)  $ in $\mathbb{R}^{3}$.

Our first theorem is the equivalence of Fourier extension and disjoint
trilinear Fourier extension.

\begin{theorem}
\label{main}The Fourier extension conjecture holds for the paraboloid
$\mathbb{P}^{2}$ in $\mathbb{R}^{3}$ \emph{if and only if} for every $q>3$
there is $\nu>0$ such that the disjoint trilinear inequality $\mathcal{E}%
_{\mathop{\rm disj}\nu}\left(  \otimes_{3}L^{q}\rightarrow L^{\frac{q}{3}%
};\varepsilon\right)  $ holds for all $\varepsilon>0$. More generally, the
following statements are equivalent:

\begin{enumerate}
\item $\mathcal{E}\left(  \otimes_{1}L^{q}\rightarrow L^{q}\right)  $ for all
$q>3$,

\item For every $q>3$ there is $\nu>0$ such that $\mathcal{E}_{\mathop{\rm
disj}\nu}\left(  \otimes_{3}L^{q}\rightarrow L^{\frac{q}{3}};\varepsilon
\right)  $ for all $\varepsilon>0$,

\item For every $q>3$ there is $\nu>0$ such that $\mathcal{E}_{\mathop{\rm
disj}\nu}\left(  \otimes_{3}L^{\infty}\rightarrow L^{\frac{q}{3}}%
;\varepsilon\right)  $ for all $\varepsilon>0$.
\end{enumerate}
\end{theorem}

The proof shows that we can take $\nu=O\left(  2^{-\frac{3q}{q-3}}\right)  $,
see (\ref{nu}). We exploit the fact that the argument in \cite[See the Remark
after (6.19) on page 1248.]{BoGu} simply requires a trilinear inequality for a
sufficiently small separation constant $\nu>0$, depending only on $q>3$. This
theorem also gives a characterization of the B\^{o}chner-Riesz inequality on
the paraboloid (and also on quadratic surfaces of positive Gaussian curvature
as mentioned above) in view of the equivalence of Fourier extension and
B\^{o}chner-Riesz for such surfaces, see Carbery \cite{Car} and \cite{Tao} for this.

\begin{remark}
It is not clear that the Fourier extension inequality for the \emph{sphere}
can be characterized by the methods used here, since when considering
(\ref{FEC}) for a surface $S$ without invariance under a quadratic
dilation,\ such as the sphere, the proof of Case 2 below requires introduction
of perturbations $S^{\prime}$ of the surface $S$ as well, which then prevents
a characterization of (\ref{FEC}) in terms of (\ref{concl thm}) - see
\cite[page 297]{TaVaVe} where this is discussed. Nevertheless, we conjecture
that the methods used here can be adapted to prove that (\ref{FEC}) holds for
\emph{all} surfaces $S$ in $\mathbb{R}^{3}$ of positive Gaussian curvature
that are bounded by some fixed constant $A$, if and only if (\ref{concl thm})
holds uniformly for the same class of surfaces\footnote{the reason being that
the parabolic rescalings used in the proof do not exit the class of surfaces
under consideration.}.
\end{remark}

Theorem \ref{main} can be put into context as follows. In the special case
where the patches $\Phi\left(  U_{1}\right)  ,\Phi\left(  U_{2}\right)
,\Phi\left(  U_{3}\right)  $ are $\nu$-transverse, then the trilinear
inquality (\ref{concl thm}) is the trilinear inequality proved by Bennett,
Carbery and Tao \cite{BeCaTa}. In the more general case when the patches
$\Phi\left(  U_{1}\right)  ,\Phi\left(  U_{2}\right)  ,\Phi\left(
U_{3}\right)  $ are merely assumed $\nu$-disjoint, then the trilinear
inequality (\ref{concl thm}) implies the Fourier extension conjecture.

\subsection{Easy directions of the proof}

The implication $\left(  1\right)  \Longrightarrow\left(  2\right)  $ of
Theorem \ref{main} follows from applying H\"{o}lder's inequality with
exponents $\frac{1}{3},\frac{1}{3},\frac{1}{3}$ to show that (\ref{FEC})
implies (\ref{concl thm}) with $\varepsilon=0$ and even without the $\nu
$-disjoint condition (\ref{weak sep}):%
\begin{align*}
&  \left(  \int_{\mathbb{R}^{3}}\left\vert \mathcal{E}f_{1}\left(  \xi\right)
\ \mathcal{E}f_{2}\left(  \xi\right)  \ \mathcal{E}f_{3}\left(  \xi\right)
\right\vert ^{\frac{q}{3}}d\xi\right)  ^{\frac{3}{q}}\lesssim\left(
\int_{\mathbb{R}^{3}}\left\vert \mathcal{E}f_{1}\left(  \xi\right)
\right\vert ^{q}d\xi\right)  ^{\frac{1}{q}}\left(  \int_{\mathbb{R}^{3}%
}\left\vert \mathcal{E}f_{2}\left(  \xi\right)  \right\vert ^{q}d\xi\right)
^{\frac{1}{q}}\left(  \int_{\mathbb{R}^{3}}\left\vert \mathcal{E}f_{3}\left(
\xi\right)  \right\vert ^{q}d\xi\right)  ^{\frac{1}{q}}\\
&  \lesssim\left(  \int_{U}\left\vert f_{1}\left(  x\right)  \right\vert
^{q}dx\right)  ^{\frac{1}{q}}\left(  \int_{U}\left\vert f_{2}\left(  x\right)
\right\vert ^{q}dx\right)  ^{\frac{1}{q}}\left(  \int_{U}\left\vert
f_{3}\left(  x\right)  \right\vert ^{q}dx\right)  ^{\frac{1}{q}}=\left\Vert
f_{1}\right\Vert _{L^{q}}\left\Vert f_{2}\right\Vert _{L^{q}}\left\Vert
f_{3}\right\Vert _{L^{q}}\ .
\end{align*}
The implication $\left(  2\right)  \Longrightarrow\left(  3\right)  $ follows
from the embedding $\left\Vert f_{k}\right\Vert _{L^{q}\left(  U_{k}\right)
}\leq\left\Vert f_{k}\right\Vert _{L^{\infty}\left(  U_{k}\right)  }\left\vert
U_{k}\right\vert ^{\frac{1}{q}}\leq\left\Vert f_{k}\right\Vert _{L^{\infty
}\left(  U_{k}\right)  }$.

The next section of this paper is devoted to an adaptation of the argument of
Bourgain and Guth \cite[Section 2]{BoGu} that will show that $\mathcal{E}%
_{\mathop{\rm disj}\nu}\left(  \otimes L^{\infty}\rightarrow L^{\frac{q}{3}%
};\varepsilon\right)  $ for all $q>3$ and $\varepsilon>0$, implies the Fourier
extension conjecture $\mathcal{E}\left(  \otimes_{1}L^{q}\rightarrow
L^{\frac{q}{3}};\varepsilon\right)  $ for $q>3$, i.e. (\ref{FEC}) in
$\mathbb{R}^{3}$, thereby establishing the implication $\left(  3\right)
\Longrightarrow\left(  1\right)  $, and completing the proof of Theorem
\ref{main}. In the third section we obtain the equivalence of the Fourier
extension conjecture with an Alpert disjoint trilinear conjecture, and in the
final appendix section, we sketch an alternate proof of the
\emph{probabilistic} Fourier extension theorem in \cite{Saw7}.

\section{Proof that disjoint trilinear extension implies Fourier extension}

A natural approach to proving $\left(  3\right)  \Longrightarrow\left(
1\right)  $ is to write $f=\sum_{K\in\mathcal{G}_{\lambda}\left[  U\right]
}\mathbf{1}_{K}f$, where $\nu=2^{-\lambda}$ and the squares $K\in
\mathcal{G}_{\lambda}\left[  U\right]  $ tile $U$ and have side length
$2^{-\lambda}$. Then we have
\begin{align*}
\left\Vert \mathcal{E}f\right\Vert _{L^{q}\left(  B\left(  0,R\right)
\right)  }^{q}  &  =\left\Vert \left(  \sum_{K\in\mathcal{G}_{\lambda}\left[
U\right]  }\mathcal{E}\left(  \mathbf{1}_{K}f\right)  \right)  ^{3}\right\Vert
_{L^{\frac{q}{3}}\left(  B\left(  0,R\right)  \right)  }^{\frac{q}{3}%
}=\left\Vert \sum_{\left(  K_{1},K_{2},K_{3}\right)  \in\mathcal{G}_{\lambda
}\left[  U\right]  ^{3}}\mathcal{E}\left(  \mathbf{1}_{K_{1}}f\right)
\ \mathcal{E}\left(  \mathbf{1}_{K_{2}}f\right)  \ \mathcal{E}\left(
\mathbf{1}_{K_{3}}f\right)  \right\Vert _{L^{\frac{q}{3}}\left(  B\left(
0,R\right)  \right)  }^{\frac{q}{3}}\\
&  =\left\Vert \left\{  \sum_{\left(  K_{1},K_{2},K_{3}\right)  \in\Gamma_{1}%
}+\sum_{\left(  K_{1},K_{2},K_{3}\right)  \in\Gamma_{2}}+\sum_{\left(
K_{1},K_{2},K_{3}\right)  \in\Gamma_{3}}\right\}  \mathcal{E}\left(
\mathbf{1}_{K_{1}}f\right)  \ \mathcal{E}\left(  \mathbf{1}_{K_{2}}f\right)
\ \mathcal{E}\left(  \mathbf{1}_{K_{3}}f\right)  \right\Vert _{L^{\frac{q}{3}%
}\left(  B\left(  0,R\right)  \right)  }^{\frac{q}{3}}\\
&  \lesssim\sum_{\alpha=1}^{3}\left\Vert \sum_{\left(  K_{1},K_{2}%
,K_{3}\right)  \in\Gamma_{\alpha}}\mathcal{E}\left(  \mathbf{1}_{K_{1}%
}f\right)  \ \mathcal{E}\left(  \mathbf{1}_{K_{2}}f\right)  \ \mathcal{E}%
\left(  \mathbf{1}_{K_{3}}f\right)  \right\Vert _{L^{\frac{q}{3}}\left(
B\left(  0,R\right)  \right)  }^{\frac{q}{3}}\equiv\sum_{\alpha=1}%
^{3}T_{\alpha}\ ,
\end{align*}
where%
\begin{align*}
\Gamma_{1}  &  \equiv\left\{  \left(  K_{1},K_{2},K_{3}\right)  \in
\mathcal{G}_{\lambda}\left[  U\right]  ^{3}:\text{no pair of squares
touch}\right\}  ,\\
\Gamma_{2}  &  \equiv\left\{  \left(  K_{1},K_{2},K_{3}\right)  \in
\mathcal{G}_{\lambda}\left[  U\right]  ^{3}:\text{for exactly one pair of
squares touch}\right\}  ,\\
\Gamma_{3}  &  \equiv\left\{  \left(  K_{1},K_{2},K_{3}\right)  \in
\mathcal{G}_{\lambda}\left[  U\right]  ^{3}:\text{every square touches
another}\right\}  .
\end{align*}

Term $T_{1}$ can be controlled by $C_{\varepsilon,\nu,\infty,q}R^{\varepsilon
}$ using (\ref{concl thm}). Term $T_{3}$ can be controlled using parabolic
rescaling as in \textbf{Case 2} below,%

\begin{align*}
&  \left\Vert \sum_{\left(  K_{1},K_{2},K_{3}\right)  \in\Gamma_{3}%
}\mathcal{E}\left(  \mathbf{1}_{K_{1}}f\right)  \ \mathcal{E}\left(
\mathbf{1}_{K_{2}}f\right)  \ \mathcal{E}\left(  \mathbf{1}_{K_{3}}f\right)
\right\Vert _{L^{\frac{q}{3}}\left(  B\left(  0,R\right)  \right)  }^{\frac
{q}{3}}\lesssim\left\Vert \sum_{K\in\mathcal{G}_{\lambda}\left[  U\right]
}\mathcal{E}\left(  \mathbf{1}_{K}f\right)  ^{3}\right\Vert _{L^{\frac{q}{3}%
}\left(  B\left(  0,R\right)  \right)  }^{\frac{q}{3}}\\
&  \leq\left(  \#\mathcal{G}_{\lambda}\left[  U\right]  \right)  ^{\frac{q}%
{3}-1}\sum_{K\in\mathcal{G}_{\lambda}\left[  U\right]  }\int_{B\left(
0,R\right)  }\left\vert \mathcal{E}\left(  \mathbf{1}_{K}f\right)  \left(
\xi\right)  \right\vert ^{q}d\xi\lesssim\left(  \#\mathcal{G}_{\lambda}\left[
U\right]  \right)  ^{\frac{q}{3}-1}\sum_{K\in\mathcal{G}_{\lambda}\left[
U\right]  }\left(  2^{-\lambda}\right)  ^{\left(  2q-4\right)  }%
\sup_{\left\Vert f\right\Vert _{L^{\infty}}\leq1}\left\Vert \mathcal{E}%
f\right\Vert _{L^{q}\left(  B\left(  0,2^{-\lambda}R\right)  \right)  }^{q}\\
&  \leq\left(  \#\mathcal{G}_{\lambda}\left[  U\right]  \right)  ^{\frac{q}%
{3}}\left(  2^{-\lambda}\right)  ^{2q-4}\sup_{\left\Vert f\right\Vert
_{L^{\infty}}\leq1}\left\Vert \mathcal{E}f\right\Vert _{L^{q}\left(  B\left(
0,2^{-\lambda}R\right)  \right)  }^{q}=\left(  2^{-\lambda}\right)  ^{\frac
{4}{3}q-4}\sup_{\left\Vert f\right\Vert _{L^{\infty}}\leq1}\left\Vert
\mathcal{E}f\right\Vert _{L^{q}\left(  B\left(  0,2^{-\lambda}R\right)
\right)  }^{q},
\end{align*}
which is at most $\frac{1}{2}\sup_{\left\Vert f\right\Vert _{L^{\infty}}\leq
1}\left\Vert \mathcal{E}f\right\Vert _{L^{q}\left(  B\left(  0,R\right)
\right)  }^{q}$ if $\lambda$ is chosen sufficiently large depending on $q>3$.
Then we have%
\begin{align*}
\sup_{\left\Vert f\right\Vert _{L^{\infty}}\leq1}\left\Vert \mathcal{E}%
f\right\Vert _{L^{q}\left(  B\left(  0,R\right)  \right)  }^{q}  &  \leq
T_{1}+T_{2}+T_{3}\leq C_{\varepsilon,\nu,\infty,q}R^{\varepsilon}+T_{2}%
+\frac{1}{2}\sup_{\left\Vert f\right\Vert _{L^{\infty}}\leq1}\left\Vert
\mathcal{E}f\right\Vert _{L^{q}\left(  B\left(  0,R\right)  \right)  }^{q},\\
&  \Longrightarrow\sup_{\left\Vert f\right\Vert _{L^{\infty}}\leq1}\left\Vert
\mathcal{E}f\right\Vert _{L^{q}\left(  B\left(  0,R\right)  \right)  }^{q}\leq
C_{\varepsilon,\nu,\infty,q}R^{\varepsilon}+T_{2}\ ,
\end{align*}
but unfortunately, term $T_{2}$ is problematic since the same argument
produces a larger power $\left(  \#\mathcal{G}_{\lambda}\left[  U\right]
\right)  ^{\frac{q}{3}+1}$ due to summing over \emph{two} independent squares
in $\mathcal{G}_{\lambda}\left[  U\right]  $. The resulting estimate $\left(
2^{-\lambda}\right)  ^{\left(  2q-8\right)  }\sup_{\left\Vert f\right\Vert
_{L^{\infty}}\leq1}\left\Vert \mathcal{E}f\right\Vert _{L^{q}\left(  B\left(
0,R\right)  \right)  }^{q}$ cannot be absorbed unless $q>4$.

Here we will use the \emph{disjoint} trilinear estimate $\mathcal{E}%
_{\mathop{\rm disj}\nu}\left(  \otimes_{3}L^{\infty}\rightarrow L^{\frac{q}%
{3}};\varepsilon\right)  $ to essentially eliminate the difficult \textbf{Case
3} of the Bourgain and Guth argument in \cite[Section 2]{BoGu}, along with the
restriction $p>\frac{10}{3}$ there. This results in an optimal local linear
inequality, which in turn proves the Fourier extension conjecture in three
dimensions by Nikishin-Maurey-Pisier theory and $\varepsilon$-removal techniques.

Suppose $S$ is a compact smooth hypersurface contained in $\mathbb{R}^{3}$
that is contained in the paraboloid $\mathbb{P}^{2}$, and denote surface
measure on $S$ by $\sigma$. The next definition is specialized from
\cite{BoGu}.

\begin{definition}
For $1<q<\infty$ and $R>0$ define $Q_{R}^{\left(  q\right)  }$ to be the best
constant in the local linear Fourier extension inequality,%
\[
\left(  \int_{B\left(  0,R\right)  }\left\vert \widehat{\Phi_{\ast}f}\left(
\xi\right)  \right\vert ^{q}d\xi\right)  ^{\frac{1}{q}}\leq Q_{R}^{\left(
q\right)  }\left\Vert f\right\Vert _{L^{\infty}\left(  U\right)  }\ ,
\]
i.e.
\begin{equation}
Q_{R}^{\left(  q\right)  }\equiv\sup_{\left\Vert f\right\Vert _{L^{\infty
}\left(  U\right)  }\leq1}\left(  \int_{B\left(  0,R\right)  }\left\vert
\mathcal{E}f\left(  \xi\right)  \right\vert ^{q}d\xi\right)  ^{\frac{1}{q}%
}=\left\Vert \mathcal{E}\right\Vert _{L^{\infty}\left(  U\right)  \rightarrow
L^{q}\left(  B\left(  0,R\right)  \right)  }.\label{Q_R}%
\end{equation}

\end{definition}

\begin{theorem}
\label{Loc lin}Let $S$ be as above. Suppose that $q>3$ and $0<\nu\leq\frac
{1}{2}2^{10}2^{-\frac{3q}{q-3}}$. If $\mathcal{E}_{\mathop{\rm disj}\nu
}\left(  \otimes_{3}L^{\infty}\rightarrow L^{\frac{q}{3}};\varepsilon\right)
$ holds for all $\varepsilon>0$, then%
\[
Q_{R}^{\left(  q\right)  }\leq C_{\varepsilon,\nu,q}R^{\varepsilon
},\ \ \ \ \ \text{for all }\varepsilon>0\text{ and }R\geq1.
\]

\end{theorem}

Using Theorem \ref{Loc lin} together with $\varepsilon$-removal techniques and
factorization theory, we can now prove that the Fourier extension conjecture
holds in three dimensions if for every $\varepsilon>0$ there is $\nu>0$ such
that $\mathcal{E}_{\mathop{\rm disj}\nu}\left(  \otimes_{3}L^{\infty
}\rightarrow L^{\frac{q}{3}};\varepsilon\right)  $ holds for all
$\varepsilon>0$.

\begin{proof}
[Proof of the implication $\left(  3\right)  \Longrightarrow\left(  1\right)
$ in Theorem \ref{main}]Statement $\left(  3\right)  $ of Theorem \ref{main}
implies the conclusion of Theorem \ref{Loc lin} for all $q>3$, which says that
the extension operator $\mathcal{E}$ maps $L^{\infty}\left(  \sigma\right)  $
to $L^{q}\left(  B\left(  0,R\right)  \right)  $ with bound $CR^{\varepsilon}$
for all $q>3$ and $\varepsilon>0$, i.e.%
\[
\left\Vert \mathcal{E}f\right\Vert _{L^{q}\left(  B\left(  0,R\right)
\right)  }\lesssim R^{\varepsilon}\left\Vert f\right\Vert _{L^{\infty}\left(
\sigma\right)  },\ \ \ \ \ \text{for all }q>3\text{ and }\varepsilon>0.
\]
By duality, this is equivalent to the restriction inequality,%
\[
\left\Vert \widehat{f}\mid_{\mathbb{P}^{2}}\right\Vert _{L^{1}\left(
\sigma\right)  }\lesssim R^{\varepsilon}\left\Vert f\right\Vert _{L^{q^{\prime
}}\left(  B\left(  0,R\right)  \right)  },\ \ \ \ \ \text{for all }q^{\prime
}<\frac{3}{2}\text{ and }\varepsilon>0.
\]
An immediate consequence of the variant \cite[Lemma A1]{BoGu} of Tao's
$\varepsilon$-removal theorem \cite[Theorem 1.2]{Tao}, is that this inequality
implies the global restriction inequality,%
\[
\left\Vert \widehat{f}\mid_{\mathbb{P}^{2}}\right\Vert _{L^{1}\left(
\sigma\right)  }\lesssim\left\Vert f\right\Vert _{L^{q^{\prime}}\left(
\mathbb{R}^{3}\right)  },\ \ \ \ \ \text{for all }q^{\prime}<\frac{3}{2},
\]
which by duality is the global extension inequality,%
\[
\left\Vert \mathcal{E}f\right\Vert _{L^{q}\left(  \mathbb{R}^{3}\right)
}\lesssim\left\Vert f\right\Vert _{L^{\infty}\left(  \sigma\right)
},\ \ \ \ \ \text{for all }q>3.
\]

In order to extend the domain $L^{\infty}\left(  \sigma\right)  $ of
$\mathcal{E}$\ to the larger space $L^{q}\left(  \sigma\right)  $, we appeal
to Nikishin-Maurey-Pisier factorization theory and interpolation. For example,
from \cite[Corollary 1.4 and Remark 1.5]{Bus} and Theorem \ref{Loc lin}, we
conclude that,%
\[
\left\Vert \mathcal{E}g\right\Vert _{L^{q}\left(  \mathbb{R}^{3}\right)
}=\left\Vert \widehat{gd\sigma}\right\Vert _{L^{q}\left(  \mathbb{R}%
^{3}\right)  }\leq C_{q,\gamma}\left\Vert g\right\Vert _{L^{q}\left(
\sigma\right)  },\ \ \ \ \ \text{for all }q>3.
\]
This completes the proof of Theorem \ref{main} assuming that Theorem
\ref{Loc lin} holds.
\end{proof}

Now we turn to proving the local linear Fourier inequality in Theorem
\ref{Loc lin} via part of the argument of Bourgain and Guth \cite[Section
2]{BoGu}, but using the $\nu$-disjoint assumption in (\ref{concl thm}) to
simplify the problematic \textbf{Case 3} of their argument.

\subsection{The pigeonholing argument of Bourgain and Guth\label{Sub pigeon}}

\begin{proof}
[Proof of Theorem \ref{Loc lin}]We begin the argument exactly as in
\cite{BoGu}, but with some changes in notation. We let the surface $S$ be a
compact smooth piece of the paraboloid $\mathbb{P}^{2}$ given by
$z_{3}=\left\vert z^{\prime}\right\vert ^{2}=z_{1}^{2}+z_{2}^{2}$ in
$\mathbb{R}^{3}$, and for $f\in L^{\infty}\left(  S\right)  $ with $\left\Vert
f\right\Vert _{L^{\infty}\left(  S\right)  }=1$, we consider the oscillatory
integral $\mathcal{E}f\left(  \xi\right)  $, which we write as%
\begin{align*}
Tf\left(  \xi\right)   &  \equiv\int_{U}e^{i\phi\left(  \xi,y\right)
}f\left(  y\right)  dy=\int_{U}e^{i\left\{  \xi_{1}\cdot y_{1}+\xi_{2}\cdot
y_{2}+\xi_{3}\left(  y_{1}^{2}+y_{2}^{2}\right)  \right\}  }f\left(  y\right)
dy\\
&  =\int_{U}e^{i\xi\cdot\left(  y,\left\vert y\right\vert ^{2}\right)
}f\left(  y\right)  dy=\widehat{f^{\Phi}}\left(  \xi\right)  =\widehat
{\Phi_{\ast}\left[  f\left(  y\right)  dy\right]  }\left(  \xi\right)
,\ \ \ \ \ \text{for }\xi\in\mathbb{R}^{3},
\end{align*}
where
\[
\phi\left(  \xi,y\right)  =\xi\cdot\Phi\left(  y\right)  \text{ and }%
\Phi\left(  y\right)  \equiv\left(  y_{1},y_{2},y_{1}^{2}+y_{2}^{2}\right)  .
\]
For $\lambda\geq1$, let $f=\sum_{I\in\mathcal{G}_{\lambda}\left[  U\right]
}\boldsymbol{1}_{I}f=\sum_{I\in\mathcal{G}_{\lambda}\left[  U\right]  }f_{I}$
and write%
\[
Tf\left(  \xi\right)  =\sum_{I\in\mathcal{G}_{\lambda}\left[  U\right]  }%
\int_{S}e^{i\xi\cdot\left(  y,\left\vert y\right\vert ^{2}\right)  }%
f_{I}\left(  y\right)  dy=\sum_{I\in\mathcal{G}_{\lambda}\left[  U\right]
}e^{i\phi\left(  \xi,c_{I}\right)  }\int e^{i\left\{  \phi\left(
\xi,y\right)  -\phi\left(  \xi,c_{I}\right)  \right\}  }f_{I}\left(  y\right)
dy=\sum_{I\in\mathcal{G}_{\lambda}\left[  S\right]  }e^{i\phi\left(  \xi
,c_{I}\right)  }T_{I}f\left(  \xi\right)  ,
\]
where
\begin{align}
T_{I}f\left(  \xi\right)   &  \equiv\int e^{i\left\{  \phi\left(
\xi,y\right)  -\phi\left(  \xi,c_{I}\right)  \right\}  }f_{I}\left(  y\right)
dy=e^{-i\left\{  \xi\cdot\Phi\left(  c_{I}\right)  \right\}  }\int
e^{i\xi\cdot\Phi\left(  y\right)  }f_{I}\left(  y\right)  dy\label{def T_I}\\
&  =e^{-i\left\{  \xi\cdot\Phi\left(  c_{I}\right)  \right\}  }\widehat
{f_{I}^{\Phi}}\left(  \xi\right)  =\widehat{\tau_{-\Phi\left(  c_{I}\right)
}f_{I}^{\Phi}}\left(  \xi\right)  ,\nonumber
\end{align}
where $\tau_{\Phi\left(  c_{I}\right)  }g\left(  z\right)  \equiv g\left(
z-\Phi\left(  c_{I}\right)  \right)  $ is translation of a function $g$ by the
vector $\Phi\left(  c_{I}\right)  $.

Note that%
\[
\left\vert \nabla_{\xi}\left\{  \xi\cdot\left(  y-c_{I},\left\vert
y\right\vert ^{2}-\left\vert c_{I}\right\vert ^{2}\right)  \right\}
\right\vert =\left\vert \left(  y-c_{I},\left\vert y\right\vert ^{2}%
-\left\vert c_{I}\right\vert ^{2}\right)  \right\vert \lesssim\frac
{1}{2^{\lambda}},\ \ \ \ \ \text{for }y\in I,
\]
implies%
\begin{align}
\nabla_{\xi}T_{I}f\left(  \xi\right)   &  =\nabla_{\xi}\int e^{i\xi
\cdot\left(  y-c_{I},\left\vert y-c_{I}\right\vert ^{2}\right)  }f_{I}\left(
y\right)  dy=\int\nabla_{\xi}e^{i\xi\cdot\left(  y-c_{I},\left\vert
y-c_{I}\right\vert ^{2}\right)  }f_{I}\left(  y\right)
dy\label{pre note that}\\
&  =\int ie^{i\xi\cdot\left(  y-c_{I},\left\vert y-c_{I}\right\vert
^{2}\right)  }\nabla_{\xi}\left\{  \xi\cdot\left(  y-c_{I},\left\vert
y\right\vert ^{2}-\left\vert c_{I}\right\vert ^{2}\right)  \right\}
f_{I}\left(  y\right)  dy,\nonumber
\end{align}
which implies%
\begin{equation}
\left\vert \nabla_{\xi}T_{I}f\left(  \xi\right)  \right\vert \leq
\int\left\vert \nabla_{\xi}\left\{  \xi\cdot\left(  y-c_{I},\left\vert
y\right\vert ^{2}-\left\vert c_{I}\right\vert ^{2}\right)  \right\}
\right\vert \left\vert f_{I}\left(  y\right)  \right\vert dy\lesssim\frac
{1}{2^{\lambda}}\left\Vert f_{I}\right\Vert _{L^{1}\left(  U\right)  }%
\lesssim\frac{1}{2^{3\lambda}}\left\Vert f_{I}\right\Vert _{L^{\infty}\left(
U\right)  },\label{note that}%
\end{equation}
since $\ell\left(  I\right)  =\frac{1}{2^{\lambda}}$. We will use the
estimates (\ref{pre note that}) and (\ref{note that}) in (\ref{exclaim'}) below.

Now let $\rho$ be a smooth rapidly decreasing bump function such that
$\widehat{\rho}\left(  \xi\right)  =1$ for $\left\vert \xi\right\vert \leq1$,
and set%
\[
\rho_{\lambda}\left(  z\right)  \equiv\frac{1}{2^{3\lambda}}\rho\left(
\frac{z}{2^{\lambda}}\right)  ,\ \ \ \ \ \widehat{\rho_{\lambda}}\left(
\xi\right)  =\widehat{\rho}\left(  2^{\lambda}\xi\right)  =1\text{ on
}B\left(  0,2^{-\lambda}\right)  \text{ and }\rho_{\lambda}\left(  z\right)
\approx\frac{1}{2^{3\lambda}}\text{ on }B\left(  0,2^{\lambda}\right)  .
\]
Then from (\ref{def T_I}) we obtain%
\[
T_{I}f\left(  \xi\right)  =T_{I}f\ast\rho_{\lambda}\left(  \xi\right)
\ ,\ \ \ \ \ \text{for }I\in\mathcal{G}_{\lambda}\left[  S\right]  \text{ and
}\xi\in\mathbb{R}^{3},
\]
since $\tau_{-\Phi\left(  c_{I}\right)  }f_{I}^{\Phi}\subset B\left(
0,2^{-\lambda}\right)  $ and $\widehat{T_{I}f}\left(  z\right)  =\tau
_{-\Phi\left(  c_{I}\right)  }f_{I}^{\Phi}\left(  z\right)  $ imply%
\[
\widehat{T_{I}f\ast\rho_{\lambda}}\left(  z\right)  =\widehat{T_{I}f}\left(
z\right)  \widehat{\rho_{\lambda}}\left(  z\right)  =\tau_{-\Phi\left(
c_{I}\right)  }f_{I}^{\Phi}\left(  z\right)  \widehat{\rho_{\lambda}}\left(
z\right)  =\tau_{-\Phi\left(  c_{I}\right)  }f_{I}^{\Phi}\left(  z\right)
=\widehat{T_{I}f}\left(  z\right)  .
\]

Now fix a point $a\in\Gamma_{\lambda}\left(  R\right)  $, where%
\[
\Gamma_{\lambda}\left(  R\right)  \equiv2^{\lambda}\mathbb{Z}^{3}\cap
B_{R},\ \ \ \ \ \text{and }B_{R}\equiv B\left(  0,R\right)  ,
\]
and restrict $\xi$ to the ball $B\left(  a,2^{\lambda}\right)  $. We will
write $B_{R}$ as apposed to $B\left(  0,R\right)  $ above in order to
emphasize the nature of the different roles played by $\lambda$ and $R$,
namely $R\nearrow\infty$ while $\lambda$ remains a fixed sufficiently large
integer to be chosen.

Then for $\xi\in B\left(  a,2^{\lambda}\right)  $ and $I\in\mathcal{G}%
_{\lambda}\left[  S\right]  $ we have
\begin{align}
&  \left\vert T_{I}f\left(  \xi\right)  \right\vert =\left\vert T_{I}f\ast
\rho_{\lambda}\left(  \xi\right)  \right\vert =\left\vert \int_{\mathbb{R}%
^{3}}T_{I}f\left(  z\right)  \rho_{\lambda}\left(  \xi-z\right)  dz\right\vert
\label{exclaim}\\
&  \leq\int_{\mathbb{R}^{3}}\left\vert T_{I}f\left(  z\right)  \right\vert
\left\vert \rho_{\lambda}\left(  \xi-z\right)  \right\vert dz\leq
\int_{\mathbb{R}^{3}}\left\vert T_{I}f\left(  z\right)  \right\vert
\sup_{\omega\in B\left(  a,2^{\lambda}\right)  }\left\vert \rho_{\lambda
}\left(  z-\omega\right)  \right\vert dz=\int_{\mathbb{R}^{3}}\left\vert
T_{I}f\left(  z\right)  \right\vert \zeta_{\lambda}\left(  z-a\right)
dz,\nonumber
\end{align}
where $\zeta_{\lambda}\left(  w\right)  \equiv\sup_{\omega-a\in B\left(
0,2^{\lambda}\right)  }\left\vert \rho_{\lambda}\left(  w-a-\omega\right)
\right\vert $, since $\rho$ can be chosen radial and,
\begin{align*}
\sup_{\omega\in B\left(  a,2^{\lambda}\right)  }\left\vert \rho_{\lambda
}\left(  \omega-z\right)  \right\vert  &  =\frac{1}{2^{3\lambda}}\sup
_{\omega\in B\left(  a,2^{\lambda}\right)  }\left\vert \rho\left(
\frac{\left(  \omega-a\right)  -\left(  z-a\right)  }{2^{\lambda}}\right)
\right\vert =\frac{1}{2^{3\lambda}}\sup_{\gamma\in B\left(  0,1\right)
}\left\vert \rho\left(  \frac{z-a}{2^{\lambda}}-\gamma\right)  \right\vert
=\zeta_{\lambda}\left(  z-a\right)  ,\\
\text{where }\zeta\left(  w\right)   &  \equiv\sup_{\left\vert w-w^{\prime
}\right\vert \leq1}\left\vert \rho\left(  w^{\prime}\right)  \right\vert .
\end{align*}
Now for $I\in\mathcal{G}_{\lambda}\left[  S\right]  $ define the right hand
side of (\ref{exclaim}) to be
\begin{align*}
&  w_{I}^{a}\left(  f\right)  \equiv\int_{\mathbb{R}^{3}}\left\vert
T_{I}f\left(  z\right)  \right\vert \zeta_{\lambda}\left(  z-a\right)
dz=\int_{\mathbb{R}^{3}}\left\vert T_{I}f\left(  z\right)  \right\vert
\zeta\left(  \frac{z-a}{2^{\lambda}}\right)  \frac{dz}{2^{3\lambda}}\\
&  =\int_{\mathbb{R}^{3}}\left\vert \widehat{f_{I}^{\Phi}}\left(  z\right)
\right\vert \zeta\left(  \frac{z-a}{2^{\lambda}}\right)  \frac{dz}%
{2^{3\lambda}}\approx\frac{1}{\left\vert B\left(  a,2^{\lambda}\right)
\right\vert }\int_{B\left(  a,2^{\lambda}\right)  }\left\vert \widehat
{f_{I}^{\Phi}}\left(  z\right)  \right\vert \ ,
\end{align*}
and refer to $w_{I}^{a}\left(  f\right)  $ as the `weight' of $\widehat
{f_{I}^{\Phi}}$ relative to the ball $B\left(  a,2^{\lambda}\right)  $, which
represents that portion of the integral of $\left\vert \widehat{f_{I}^{\Phi}%
}\left(  z\right)  \right\vert $ that is taken over the ball $B\left(
a,2^{\lambda}\right)  $. Note that $w_{I}^{a}\left(  f\right)  \lesssim
\left\Vert \widehat{f_{I}^{\Phi}}\right\Vert _{L^{\infty}}\lesssim\left\Vert
\mathbf{1}_{I}f\right\Vert _{L^{1}}\leq\left\vert I\right\vert =2^{-2\lambda}$.

Summarizing, we have%
\begin{equation}
\left\vert T_{I}f\left(  \xi\right)  \right\vert \leq\int_{\mathbb{R}^{3}%
}\left\vert T_{I}f\left(  z\right)  \right\vert \zeta_{\lambda}\left(
z-a\right)  dz=w_{I}^{a}\left(  f\right)  \ ,\ \ \ \ \ \text{for }\xi\in
B\left(  a,2^{\lambda}\right)  .\label{pre exclaim'}%
\end{equation}
and
\begin{equation}
\int_{\mathbb{R}^{3}}\left\vert T_{I}f\left(  z\right)  \right\vert
\zeta_{\lambda}\left(  z-\xi\right)  dz\approx w_{I}^{a}\left(  f\right)
\text{ },\ \ \ \ \ \text{for }\xi\in B\left(  a,2^{\lambda}\right)
,\label{exclaim'}%
\end{equation}
since for $\xi\in B\left(  a,2^{\lambda}\right)  $ we have $\zeta_{\lambda
}\left(  \xi-z\right)  \approx\zeta_{\lambda}\left(  a-z\right)
=\zeta_{\lambda}\left(  z-a\right)  $ by (\ref{pre note that}) and
(\ref{note that}).

Now set%
\[
w_{\ast}^{a}\left(  f\right)  \equiv\max_{I\in\mathcal{G}_{\lambda}\left[
S\right]  }w_{I}^{a}\left(  f\right)  =\max_{I\in\mathcal{G}_{\lambda}\left[
S\right]  }\int_{\mathbb{R}^{3}}\left\vert T_{I}f\left(  z\right)  \right\vert
\zeta_{\lambda}\left(  z-a\right)  dz,
\]
and fix $I_{\ast}^{a}$ such that
\[
w_{I_{\ast}^{a}}^{a}=w_{\ast}^{a}.
\]
For $1\ll\lambda^{\prime}\ll\lambda$, and $\alpha,\beta,\gamma,\delta
\in\mathbb{N}$ chosen appropriately, we will estimate the contributions to the
norm $\left\Vert Tf\right\Vert _{L^{q}\left(  B_{R}\right)  }$ in three
exhaustive cases in turn. The first case will yield the growth factor
$R^{\varepsilon}$, while the next two cases will be absorbed. In fact, we show
at the end of the proof that we may take%
\[
\beta=1\text{, \ \ \ \ }\alpha=\delta=\gamma=2\text{, \ \ \ \ }\lambda
=2\lambda^{\prime}\text{, \ \ \ \ and }\lambda^{\prime}>\frac{3}{2}\frac
{q}{q-3}.
\]

\end{proof}

\subsection{Case 1: Separated interaction\label{Sub Case 1}}

\begin{proof}
[Proof continued]In \textbf{Case 1} we assume the following property. There
exists a triple of squares $I_{0}^{a},J_{0}^{a},K_{0}^{a}\in\mathcal{G}%
_{\lambda}\left[  U\right]  $ such that%
\begin{align*}
w_{I_{0}^{a}},w_{J_{0}^{a}},w_{K_{0}^{a}}  &  >2^{-\alpha\lambda}w_{\ast}%
^{a}\ ,\ \ \ \ \ \text{and }\left\vert \mathbf{c}_{I_{0}^{a}}-\mathbf{c}%
_{J_{0}^{a}}\right\vert ,\left\vert \mathbf{c}_{J_{0}^{a}}-\mathbf{c}%
_{K_{0}^{a}}\right\vert ,\left\vert \mathbf{c}_{K_{0}^{a}}-\mathbf{c}%
_{I_{0}^{a}}\right\vert >2^{10}2^{-\beta\lambda}\ ,\\
&  \fbox{$%
\begin{array}
[c]{ccccc}%
\mathbf{c}_{I_{0}^{a}} &  &  &  & \\
& \cdot &  &  & \\
&  & \mathbf{c}_{K_{0}^{a}} & \leftrightarrows & \mathbf{c}_{K_{0}^{a}}\\
&  &  & \cdot & \\
&  &  &  & \mathbf{c}_{J_{0}^{a}}%
\end{array}
$},
\end{align*}
i.e. there exists a `$2^{10}2^{-\beta\lambda}$-separated' triple $I_{0}%
^{a},J_{0}^{a},K_{0}^{a}$ of squares of side length $2^{-\lambda}$, such that
each of $I_{0}^{a}$, $J_{0}^{a}\ $and $K_{0}^{a}$ have near maximal weight. In
\textbf{Case 1} we will use the $\nu$-disjoint trilinear estimate in Theorem
\ref{main} with $\nu=2^{10}2^{-\beta\lambda}$, and $U_{I_{0}^{a}}$ equal to a
cube of side length $\mathop{\rm dist}\left(  I_{0}^{a},J_{0}^{a}\cup
K_{0}^{a}\right)  $, and similarly for $U_{J_{0}^{a}}$ and $U_{K_{0}^{a}}$.
For $\xi\in B\left(  a,2^{\lambda}\right)  $ we throw away the unimodular
function $e^{-i\Phi\left(  c_{I}\right)  \cdot\xi}$, and using
(\ref{pre exclaim'}), we estimate that for $\xi\in B\left(  a,2^{\lambda
}\right)  $,%
\begin{equation}
\left\vert Tf\left(  \xi\right)  \right\vert =\left\vert \sum_{L\in
\mathcal{G}_{\lambda}\left[  S\right]  }e^{i\phi\left(  \xi,c_{I}\right)
}T_{L}f\left(  \xi\right)  \right\vert \leq\sum_{L\in\mathcal{G}_{\lambda
}\left[  S\right]  }\left\vert T_{L}f\left(  \xi\right)  \right\vert
\lesssim\sum_{L\in\mathcal{G}_{\lambda}\left[  S\right]  }w_{L}^{a}%
<2^{2\lambda}w_{\ast}^{a}<2^{\left(  2+\alpha\right)  \lambda}\left(
w_{I_{0}^{a}}w_{J_{0}^{a}}w_{K_{0}^{a}}\right)  ^{\frac{1}{3}}%
,\label{T weight}%
\end{equation}
since the\emph{\ fixed} triple $\left(  I_{0}^{a},J_{0}^{a},K_{0}^{a}\right)
$ satisfies the near maximal weight condition in \textbf{Case 1}:%
\[
w_{\ast}^{a}<\min\left\{  2^{\alpha\lambda}w_{I_{0}^{a}}^{a},2^{\alpha\lambda
}w_{J_{0}^{a}}^{a},2^{\alpha\lambda}w_{K_{0}^{a}}^{a}\right\}  \leq
2^{\alpha\lambda}\left(  w_{I_{0}^{a}}^{a}\right)  ^{\frac{1}{3}}\left(
w_{J_{0}^{a}}^{a}\right)  ^{\frac{1}{3}}\left(  w_{K_{0}^{a}}^{a}\right)
^{\frac{1}{3}}.
\]
Let $\nu=2^{10}2^{-\beta\lambda}$. Then for $q>3$ and $\xi\in B\left(
a,2^{\lambda}\right)  $, we have from boundedness of $Tf$, and (\ref{T weight}%
) and (\ref{exclaim'}), followed by H\"{o}lder's inequality, that%
\begin{align*}
&  \left\vert Tf\left(  \xi\right)  \right\vert ^{q}\lesssim2^{q\left(
2+\alpha\right)  \lambda}\left(  w_{I_{0}^{a}}^{a}w_{J_{0}^{a}}^{a}%
w_{K_{0}^{a}}^{a}\right)  ^{\frac{q}{3}}\\
&  \approx2^{q\left(  2+\alpha\right)  \lambda}\left(  \int_{\mathbb{R}^{3}%
}\left\vert T_{I_{0}^{a}}f\left(  z_{1}\right)  \right\vert \zeta_{\lambda
}\left(  z_{1}-a\right)  dz_{1}\right)  ^{\frac{q}{3}}\left(  \int
_{\mathbb{R}^{3}}\left\vert T_{J_{0}^{a}}f\left(  z_{2}\right)  \right\vert
\zeta_{\lambda}\left(  z_{2}-a\right)  dz_{2}\right)  ^{\frac{q}{3}}\left(
\int_{\mathbb{R}^{3}}\left\vert T_{K_{0}^{a}}f\left(  z_{3}\right)
\right\vert \zeta_{\lambda}\left(  z_{3}-a\right)  dz_{3}\right)  ^{\frac
{q}{3}}\\
&  \lesssim2^{q\left(  2+\alpha\right)  \lambda}\int_{\mathbb{R}^{3}}%
\int_{\mathbb{R}^{3}}\int_{\mathbb{R}^{3}}\left\vert T_{I_{0}^{a}}f\left(
\xi-z_{1}\right)  T_{J_{0}^{a}}f\left(  \xi-z_{2}\right)  T_{K_{0}^{a}%
}f\left(  \xi-z_{3}\right)  \right\vert ^{\frac{q}{3}}\zeta_{\lambda}\left(
z_{1}\right)  \zeta_{\lambda}\left(  z_{2}\right)  \zeta_{\lambda}\left(
z_{3}\right)  dz_{1}dz_{2}dz_{3}\\
&  \lesssim2^{q\left(  2+\alpha\right)  \lambda}\sum_{\left(  I,J,K\right)
\in\mathcal{G}_{\lambda}^{\nu-\mathop{\rm separated}}\left[  U\right]  }%
\int_{\mathbb{R}^{9}}\left\vert T_{I}f\left(  \xi-z_{1}\right)  T_{J}f\left(
\xi-z_{2}\right)  T_{K}f\left(  \xi-z_{3}\right)  \right\vert ^{\frac{q}{3}%
}d\mu_{\lambda}\left(  z_{1},z_{2},z_{3}\right)  \ ,
\end{align*}
where
\[
\mathcal{G}_{\lambda}^{\nu-\mathop{\rm separated}}\left[  U\right]
\equiv\left\{  \left(  I,J,K\right)  \in\mathcal{G}_{\lambda}:\left(
I,J,K\right)  \text{ is }\nu\text{-separated as in (\ref{weak sep})}\right\}
,
\]
and%
\[
d\mu_{\lambda}\left(  z_{1},z_{2},z_{3}\right)  \equiv\zeta_{\lambda}\left(
z_{1}\right)  \zeta_{\lambda}\left(  z_{2}\right)  \zeta_{\lambda}\left(
z_{3}\right)  dz_{1}dz_{2}dz_{3}%
\]
is a bounded multiple of a probability measure.

Then we have%
\begin{align*}
&  \int_{B\left(  a,2^{\lambda}\right)  }\left\vert Tf\left(  \xi\right)
\right\vert ^{q}d\xi\\
&  \lesssim2^{q\left(  2+\alpha\right)  \lambda}\int_{B\left(  a,2^{\lambda
}\right)  }\sum_{\left(  I,J,K\right)  \in\mathcal{G}_{\lambda}^{\nu
-\mathop{\rm separated}}\left[  U\right]  }\int_{\mathbb{R}^{9}}\left\vert
T_{I}f\left(  \xi-z_{1}\right)  T_{J}f\left(  \xi-z_{2}\right)  T_{K}f\left(
\xi-z_{3}\right)  \right\vert ^{\frac{q}{3}}d\mu_{\lambda}\left(  z_{1}%
,z_{2},z_{3}\right)  d\xi\\
&  =2^{q\left(  2+\alpha\right)  \lambda}\sum_{\left(  I,J,K\right)
\in\mathcal{G}_{\lambda}^{\nu-\mathop{\rm separated}}\left[  U\right]  }%
\int_{\mathbb{R}^{9}}\left\{  \int_{B\left(  a,2^{\lambda}\right)  }\left\vert
T_{I}f\left(  \xi-z_{1}\right)  T_{J}f\left(  \xi-z_{2}\right)  T_{K}f\left(
\xi-z_{3}\right)  \right\vert ^{\frac{q}{3}}d\xi\right\}  d\mu_{\lambda
}\left(  z_{1},z_{2},z_{3}\right)  .
\end{align*}
Now consider those $a\in\Gamma_{\lambda}\left(  R\right)  $ for which
\textbf{Case 1} is in effect for the ball $B\left(  a,2^{\lambda}\right)  $
and denote by $\Gamma_{\lambda}\left(  \text{\textbf{Case 1}}\right)  $ the
union of all the balls $B\left(  a,2^{\lambda}\right)  $ for which $a$ is in
\textbf{Case 1}. Summing over points $a\in\Gamma_{\lambda}\left(  R\right)  $
such that \textbf{Case 1} is in effect for the ball $B\left(  a,2^{\lambda
}\right)  $, we obtain%
\begin{align*}
&  \sum_{a\in\Gamma_{\lambda}\left(  \text{\textbf{Case 1}}\right)  }%
\int_{B\left(  a,2^{\lambda}\right)  }\left\vert Tf\left(  \xi\right)
\right\vert ^{q}d\xi\\
&  \lesssim2^{q\left(  2+\alpha\right)  \lambda}\sum_{\left(  I,J,K\right)
\in\mathcal{G}_{\lambda}^{\nu-\mathop{\rm separated}}\left[  U\right]  }%
\int_{\mathbb{R}^{9}}\left\{  \sum_{a\in\Gamma_{\lambda}\left(
\text{\textbf{Case 1}}\right)  }\int_{B\left(  a,2^{\lambda}\right)
}\left\vert T_{I}f\left(  \xi-z_{1}\right)  T_{J}f\left(  \xi-z_{2}\right)
T_{K}f\left(  \xi-z_{3}\right)  \right\vert ^{\frac{q}{3}}d\xi\right\}
d\mu_{\lambda}\left(  z_{1},z_{2},z_{3}\right) \\
&  \lesssim2^{q\left(  2+\alpha\right)  \lambda}\sum_{\left(  I,J,K\right)
\in\mathcal{G}_{\lambda}^{\nu-\mathop{\rm separated}}\left[  U\right]  }%
\int_{\mathbb{R}^{9}}\left\{  \int_{B_{R}}\left\vert T_{I}f\left(  \xi
-z_{1}\right)  T_{J}f\left(  \xi-z_{2}\right)  T_{K}f\left(  \xi-z_{3}\right)
\right\vert ^{\frac{q}{3}}d\xi\right\}  d\mu_{\lambda}\left(  z_{1}%
,z_{2},z_{3}\right) \\
&  \leq2^{q\left(  2+\alpha\right)  \lambda}\left(  \#\mathcal{G}_{\lambda
}^{\nu-\mathop{\rm separated}}\left[  U\right]  \right)  \int_{\mathbb{R}^{9}%
}\left\{  C_{\varepsilon}^{q}R^{q\varepsilon}\right\}  d\mu_{\lambda}\left(
z_{1},z_{2}...,z_{N}\right)  \lesssim C_{\varepsilon}^{q}2^{q\left(
2+\alpha\right)  \lambda}2^{6\lambda}R^{q\varepsilon},
\end{align*}
upon appealing to the $\nu$-disjoint trilinear assumption $\mathcal{E}%
_{\mathop{\rm disj}\nu}\left(  \otimes_{3}L^{p}\rightarrow L^{\frac{p}{3}%
}\right)  $ in part (3) of Theorem \ref{main} with $\nu=2^{10}2^{-\lambda}$,
and%
\[
f_{1}=\widetilde{\mathsf{M}_{z_{1}}}f_{I},\ \ \ f_{2}=\widetilde
{\mathsf{M}_{z_{1}}}f_{J},\ \ \ \ f_{3}=\widetilde{\mathsf{M}_{z_{1}}}f_{K},
\]
where $\widetilde{\mathsf{M}_{z}}\left(  x\right)  =e^{i\left\langle
z,\Phi\left(  x\right)  \right\rangle }$. Indeed, if $\mathsf{M}_{z}\left(
w\right)  =e^{i\left\langle z,w\right\rangle }$ then $\mathsf{M}_{z}\Phi
_{\ast}=\Phi_{\ast}\widetilde{\mathsf{M}_{z}}$ since for $\varphi$ continuous,
and with the pushforward and pullback operators $\Phi_{\ast},\Phi^{\ast}$, we
have%
\begin{align*}
\left\langle \mathsf{M}_{z}\Phi_{\ast}g,\varphi\right\rangle  &  =\int\left\{
e^{i\left\langle z,w\right\rangle }\Phi_{\ast}g\left(  w\right)  \right\}
\varphi\left(  w\right)  dw=\int\left\{  e^{i\left\langle z,w\right\rangle
}\varphi\left(  w\right)  \right\}  \Phi_{\ast}g\left(  w\right)
dw=\int\left\{  e^{i\left\langle z,\Phi\left(  x\right)  \right\rangle
}\varphi\left(  \Phi\left(  x\right)  \right)  \right\}  g\left(  x\right)
dx\\
&  =\int\left\{  e^{i\left\langle z,\Phi\left(  x\right)  \right\rangle
}g\left(  x\right)  \right\}  \Phi^{\ast}\varphi\left(  x\right)
dx=\left\langle \widetilde{\mathsf{M}_{z}}g,\Phi^{\ast}\varphi\right\rangle
=\left\langle \Phi_{\ast}\widetilde{\mathsf{M}_{z}}g,\varphi\right\rangle .
\end{align*}
Thus,%
\begin{align*}
\left\vert T_{I}f\left(  \xi-z_{1}\right)  \right\vert  &  =\widehat
{\Phi_{\ast}f_{I}}\left(  \xi-z_{1}\right)  =\widehat{\mathsf{M}_{z_{1}}%
\Phi_{\ast}f_{I}}\left(  \xi\right)  =\widehat{\Phi_{\ast}\widetilde
{\mathsf{M}_{z_{1}}}f_{I}}\left(  \xi\right)  =\mathcal{E}f_{1}\left(
\xi\right)  ,\\
\text{and }\left\vert T_{J}f\left(  \xi-z_{2}\right)  \right\vert  &
=\mathcal{E}f_{1}\left(  \xi\right)  \text{ and }\left\vert T_{K}f\left(
\xi-z_{1}\right)  \right\vert =\mathcal{E}f_{3}\left(  \xi\right)  \text{,}%
\end{align*}
and so%
\begin{align*}
&  \int_{B_{R}}\left\vert T_{I}f\left(  \xi-z_{1}\right)  T_{J}f\left(
\xi-z_{2}\right)  T_{K}f\left(  \xi-z_{3}\right)  \right\vert ^{\frac{q}{3}%
}d\xi=\int_{B_{R}}\left\vert \mathcal{E}f_{1}\left(  \xi\right)
\mathcal{E}f_{2}\left(  \xi\right)  \mathcal{E}f_{3}\left(  \xi\right)
\right\vert ^{\frac{q}{3}}d\xi\\
&  =\left\Vert \prod_{j=1}^{3}\mathcal{E}_{j}f_{j}\right\Vert _{L^{^{\frac
{q}{3}}}\left(  B_{R}\right)  }^{\frac{q}{3}}\leq\left(  C_{\varepsilon,\nu
,q}R^{\varepsilon}\right)  ^{q}\prod_{j=1}^{3}\left\Vert f_{j}\right\Vert
_{L^{\infty}}^{\frac{q}{3}}=\left(  C_{\varepsilon,\nu,q}R^{\varepsilon
}\right)  ^{q},
\end{align*}
since the triple $\left(  I,J,K\right)  $ is $\nu$-separated as in
(\ref{weak sep}), and since $\left\vert f_{j}\right\vert \leq1$. As a
consequence we have%
\begin{align*}
&  \sum_{\left(  I,J,K\right)  \in\mathcal{G}_{\lambda}^{\nu
-\mathop{\rm separated}}\left[  U\right]  }\int_{B_{R}}\left\vert
T_{I}f\left(  \xi\right)  T_{J}f\left(  \xi\right)  T_{K}f\left(  \xi\right)
\right\vert ^{\frac{q}{3}}d\xi\\
&  \leq\sum_{\left(  I,J,K\right)  \in\mathcal{G}_{\lambda}^{\nu
-\mathop{\rm separated}}\left[  U\right]  }\left(  C_{\varepsilon,\nu
,q}R^{\varepsilon}\right)  ^{q}\lesssim2^{6\lambda}\left(  C_{\varepsilon
,\nu,q}R^{\varepsilon}\right)  ^{q}=\left(  C_{\varepsilon,\nu,q}\right)
^{q}2^{6\lambda}R^{q\varepsilon}.
\end{align*}

Altogether then, we have proved that
\[
\left\Vert \mathbf{1}_{\Gamma_{\lambda}\left(  \text{\textbf{Case 1}}\right)
}Tf\right\Vert _{L^{q}\left(  B_{R}\right)  }\lesssim C_{\varepsilon,\nu
,q}2^{\left(  \frac{6}{q}+2+\alpha\right)  \lambda}R^{\varepsilon},
\]
where $\mathbf{1}_{\Gamma_{\lambda}\left(  \text{\textbf{Case 1}}\right)  }$
indicates the union of those balls $B\left(  a,2^{\lambda}\right)  $ for which
\textbf{Case 1} holds.
\end{proof}

\subsection{Case 2: Clustered interaction\label{Sub Case 2}}

\begin{proof}
[Proof continued]In \textbf{Case 2} we assume the following property. If
$\left\vert c_{I}-c_{I_{\ast}^{a}}\right\vert >2^{-\gamma\lambda^{\prime}}$,
then $w_{I}^{a}\leq2^{-\delta\lambda}w_{\ast}^{a}$. In other words, if $I$ is
sufficiently far from $I_{\ast}^{a}$, then $w_{I}^{a}$ is much smaller than
$w_{\ast}^{a}$, i.e.%
\begin{align*}
\mathop{\rm dist}\left(  I,I_{\ast}^{a}\right)   &  >2^{-\gamma\lambda
^{\prime}}\Longrightarrow\int_{\mathbb{R}^{3}}\left\vert T_{I}f\left(
z\right)  \right\vert \zeta_{\lambda}\left(  z-a\right)  dz\leq2^{-\delta
\lambda}\int_{\mathbb{R}^{3}}\left\vert T_{I_{\ast}^{a}}f\left(  z\right)
\right\vert \zeta_{\lambda}\left(  z-a\right)  dz,\\
&  \fbox{$%
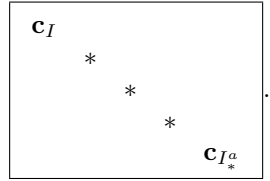
\begin{array}
[c]{ccccc}%
\mathbf{c}_{I} &  &  &  & \\
& \ast &  &  & \\
&  & \ast &  & \\
&  &  & \ast & \\
&  &  &  & \mathbf{c}_{I_{\ast}^{a}}%
\end{array}
$}.
\end{align*}
In this case we will use rescaling and recursion as in \cite{TaVaVe}.

Let $\xi\in B\left(  a,2^{\lambda}\right)  $ for some $a\in\Gamma_{\lambda
}\left(  R\right)  $. Using that $\left\vert c_{I}-c_{I_{\ast}^{a}}\right\vert
>2^{-\gamma\lambda^{\prime}}$ implies $w_{I}^{a}\leq2^{-\delta\lambda}w_{\ast
}^{a}$ in this case, we have with $\left\vert \cdot\right\vert
_{\mathop{\rm square}}$ denoting the `square' norm in $\mathbb{R}^{3}$,
\begin{align*}
\left\vert Tf\left(  \xi\right)  \right\vert  &  =\left\vert \sum
_{I\in\mathcal{G}_{\lambda}\left[  U\right]  }\int_{I}e^{i\phi\left(
\xi,y\right)  }f\left(  y\right)  dy\right\vert \\
&  \leq\left\vert \sum_{I\in\mathcal{G}_{\lambda}\left[  U\right]
:\ \left\vert c_{I}-c_{I_{\ast}^{a}}\right\vert _{\mathop{\rm square}}%
\leq2^{-\lambda^{\prime}}}\int_{I}e^{i\phi\left(  \xi,y\right)  }f\left(
y\right)  dy\right\vert +\sum_{I\in\mathcal{G}_{\lambda}\left[  U\right]
:\ \left\vert c_{I}-c_{I_{\ast}^{a}}\right\vert >2^{-\lambda^{\prime}}%
}\left\vert T_{I}f\left(  \xi\right)  \right\vert \\
&  \leq10\max_{K\in\mathcal{G}_{\lambda^{\prime}}\left[  U\right]  }\left\vert
\int_{K}e^{i\phi\left(  \xi,y\right)  }f\left(  y\right)  dy\right\vert
+\sum_{\left\vert c_{I}-c_{I_{\ast}^{a}}\right\vert >2^{-\lambda^{\prime}}%
}w_{I}^{a}\\
&  \leq10\max_{K\in\mathcal{G}_{\lambda^{\prime}}\left[  U\right]  }\left\vert
T_{K}f\left(  \xi\right)  \right\vert +\left(  \#\mathcal{G}_{\lambda}\left[
U\right]  \right)  2^{-\delta\lambda}w_{\ast}^{a}\leq10\max_{K\in
\mathcal{G}_{\lambda^{\prime}}\left[  U\right]  }\left\vert T_{K}f\left(
\xi\right)  \right\vert +2^{\left(  2-\delta\right)  \lambda}w_{\ast}^{a}\ ,
\end{align*}
since if $K_{\ast}^{a}\in\mathcal{G}_{\lambda^{\prime}}\left[  S\right]  $
contains $I_{\ast}^{a}$, then we decompose $f=f_{K_{\ast}^{a}}+\sum
_{I\in\mathcal{G}_{\lambda}\left[  S\right]  :\ I\cap K_{\ast}=\emptyset}%
f_{I}$, and without loss of generality we may also assume $\left\vert
c_{I}-c_{I_{\ast}^{a}}\right\vert \gtrsim2^{-\gamma\lambda^{\prime}}$. Now
\begin{align*}
\int\left\vert T_{I}f\left(  \xi-z\right)  \right\vert \zeta_{\lambda}%
^{a}\left(  z\right)  dz  &  \leq\left(  \int\left\vert T_{I}f\left(
\xi-z\right)  \right\vert ^{q}\zeta_{\lambda}\left(  z-a\right)  dz\right)
^{\frac{1}{q}}\left(  \int\zeta_{\lambda}\left(  z-a\right)  dz\right)
^{\frac{1}{q^{\prime}}}\\
&  \lesssim\left(  \int\left\vert T_{I}f\left(  z-a\right)  \right\vert
^{q}\zeta_{\lambda}\left(  z\right)  dz\right)  ^{\frac{1}{q}},
\end{align*}
and so for $\xi\in B\left(  a,2^{\lambda}\right)  $,
\begin{align*}
\left\vert Tf\left(  \xi\right)  \right\vert ^{q}  &  \leq C\sum
_{K\in\mathcal{G}_{\lambda^{\prime}}\left[  U\right]  }\left\vert
T_{K}f\left(  \xi\right)  \right\vert ^{q}+C2^{\left(  2-\delta\right)
\lambda q}\int\left\vert T_{I_{\ast}^{a}}f\left(  z\right)  \right\vert
^{q}\zeta_{\lambda}\left(  z-a\right)  dz\\
&  \leq C\sum_{K\in\mathcal{G}_{\lambda^{\prime}}\left[  U\right]  }\left\vert
T_{K}f\left(  \xi\right)  \right\vert ^{q}+C2^{\left(  2-\delta\right)
\lambda q}\sum_{I\in\mathcal{G}_{\lambda}\left[  U\right]  }\int\left\vert
T_{I}f\left(  z\right)  \right\vert ^{q}\zeta_{\lambda}\left(  z-a\right)  dz,
\end{align*}
where we have added in all $K\in\mathcal{G}_{\lambda^{\prime}}\left[
S\right]  $ rather than just $K_{\ast}^{a}$, and all $I\in\mathcal{G}%
_{\lambda}\left[  S\right]  $ rather than just $I_{\ast}^{a}$.

Summing over $a\in\Gamma_{\lambda}\left(  R\right)  $, we see that the
corresponding contribution over $B_{R}$ is at most%
\begin{align}
& \label{contrib Case 2}\\
\left\Vert \mathbf{1}_{\Gamma_{\lambda}\left(  \text{\textbf{Case 2}}\right)
}Tf\right\Vert _{L^{q}\left(  B_{R}\right)  }^{q} &  \equiv\sum_{a\in
\Gamma_{\lambda}\left(  R\right)  }\left\{  C\sum_{K\in\mathcal{G}%
_{\lambda^{\prime}}\left[  U\right]  }\int_{B\left(  a,2^{s}\right)
}\left\vert T_{K}f\left(  \xi\right)  \right\vert ^{q}d\xi+C2^{\left(
2-\delta\right)  \lambda q}\sum_{I\in\mathcal{G}_{\lambda}\left[  U\right]
}\int\left\vert T_{I}f\left(  z\right)  \right\vert ^{q}\zeta_{s}^{a}\left(
z\right)  dz\right\} \nonumber\\
&  \lesssim C\sum_{K\in\mathcal{G}_{\lambda^{\prime}}\left[  U\right]  }%
\int_{B_{R}}\left\vert T_{K}f\left(  \xi\right)  \right\vert ^{q}%
d\xi+C2^{-3\lambda}2^{\left(  2-\delta\right)  \lambda q}\sum_{I\in
\mathcal{G}_{\lambda}\left[  U\right]  }\int_{B_{R}}\left\vert T_{I}f\left(
\xi\right)  \right\vert ^{q}d\xi,\nonumber
\end{align}
since $\sum_{a\in\Gamma_{\lambda}\left(  R\right)  }\zeta_{\lambda}\left(
z-a\right)  \lesssim2^{-3\lambda}\mathbf{1}_{B_{R}}\left(  z\right)
+\mathop{\rm rapid}\mathop{\rm decay}$.

At this point we follow \cite{BoGu} in using parabolic rescaling, as
introduced in Tao, Vargas and Vega \cite{TaVaVe}, on the integral
\[
\mathop{\rm Int}_{\rho}\left(  \xi\right)  \equiv\int_{\left\vert
y-\overline{y}\right\vert <\rho}e^{i\phi\left(  \xi,y\right)  }f\left(
y\right)  dy=\int_{\left\vert y-\overline{y}\right\vert <\rho}e^{i\left[
\xi_{1}y_{1}+\xi_{2}y_{2}+\xi_{2}\left(  y_{1}^{2}+y_{2}^{2}\right)  \right]
}f\left(  y\right)  dy,\ \ \ \ \ \text{for }0<\rho<1,
\]
to obtain%
\begin{align*}
&  \left\vert \mathop{\rm Int}_{\rho}\left(  \xi\right)  \right\vert
\overset{y=\overline{y}+y^{\prime}}{=}\left\vert \int_{\left\vert y^{\prime
}\right\vert <\rho}e^{i\left[  \xi_{1}\left(  \overline{y}_{1}+y_{1}^{\prime
}\right)  +\xi_{2}\left(  \overline{y}_{2}+y_{2}^{\prime}\right)  +\xi
_{3}\left(  \left(  \overline{y}_{1}+y_{1}^{\prime}\right)  ^{2}+\left(
\overline{y}_{2}+y_{2}^{\prime}\right)  ^{2}\right)  \right]  }f\left(
\overline{y}+y^{\prime}\right)  dy^{\prime}\right\vert \\
&  =\left\vert \int_{\left\vert y^{\prime}\right\vert <\rho}e^{i\left[
\left(  \xi_{1}+2\overline{y_{1}}\xi_{3}\right)  y_{1}^{\prime}+\left(
\xi_{2}+2\overline{y_{2}}\xi_{3}\right)  y_{2}^{\prime}+\xi_{3}\left\vert
y^{\prime}\right\vert ^{2}\right]  }f\left(  \overline{y}+y^{\prime}\right)
dy^{\prime}\right\vert .
\end{align*}
Thus we conclude that%
\begin{align}
&  \ \ \ \ \ \ \ \ \ \ \ \ \ \ \ \ \ \ \ \ \ \ \ \ \ \ \ \ \ \ \left\Vert
\mathop{\rm Int}_{\rho}\right\Vert _{L^{q}\left(  B_{R}\right)  }=\left(
\int_{B_{R}}\left\vert \mathop{\rm Int}_{\rho}\left(  \xi\right)  \right\vert
^{q}d\xi\right)  ^{\frac{1}{q}}\label{Int est}\\
&  =\left(  \int_{B_{R}}\left\vert \int_{\left\vert y^{\prime}\right\vert
<\rho}e^{i\left[  \left(  \xi_{1}+2\overline{y_{1}}\xi_{3}\right)
y_{1}^{\prime}+\left(  \xi_{2}+2\overline{y_{2}}\xi_{3}\right)  y_{2}^{\prime
}+\xi_{3}\left\vert y^{\prime}\right\vert ^{2}\right]  }f\left(  \overline
{y}+y^{\prime}\right)  dy^{\prime}\right\vert ^{q}d\xi\right)  ^{\frac{1}{q}%
}\nonumber\\
&  =\left(  \int_{B_{R}}\left\vert \int_{\left\vert y^{\prime}\right\vert
<\rho}e^{i\left[  \left(  \rho\xi_{1}+2\overline{\frac{y_{1}}{\rho}}\rho
^{2}\xi_{3}\right)  \frac{y_{1}^{\prime}}{\rho}+\left(  \rho\xi_{2}%
+2\overline{\frac{y_{2}}{\rho}}\rho^{2}\xi_{3}\right)  y_{2}^{\prime}+\rho
^{2}\xi_{3}\left\vert \frac{y^{\prime}}{\rho}\right\vert ^{2}\right]
}f\left(  \overline{y}+y^{\prime}\right)  \rho^{2}d\left(  \frac{y^{\prime}%
}{\rho}\right)  \right\vert ^{q}\frac{d\left(  \rho\xi^{\prime}\right)
d\left(  \rho^{2}\xi_{3}\right)  }{\rho^{4}}\right)  ^{\frac{1}{q}}\nonumber\\
&  =\rho^{2}\rho^{-\frac{4}{q}}\left(  \int_{B_{\rho R}}\left\vert
\int_{\left\vert y^{\prime}\right\vert <1}e^{i\left[  \left(  \xi
_{1}+2\overline{y_{1}}\xi_{3}\right)  y_{1}^{\prime}+\left(  \xi
_{2}+2\overline{y_{2}}\xi_{3}\right)  y_{2}^{\prime}+\xi_{3}\left\vert
y^{\prime}\right\vert ^{2}\right]  }f\left(  \rho\left(  \overline
{y}+y^{\prime}\right)  \right)  dy^{\prime}\right\vert ^{q}d\xi^{\prime}%
d\xi_{3}\right)  ^{\frac{1}{q}}\leq C\rho^{2}\rho^{-\frac{4}{q}}Q_{\rho
R}^{\left(  q\right)  }\ ,\nonumber
\end{align}
where $Q_{\rho R}^{\left(  q\right)  }\leq Q_{R}^{\left(  q\right)  }$ is
defined in (\ref{Q_R}), since%
\begin{align}
&  \text{the }L^{\infty}\text{ norm of }f\text{ is unchanged by dilation, and
since}\label{invariant}\\
&  \text{the paraboloid is invariant under parabolic rescaling.}\nonumber
\end{align}
\newline Note that the factor $\rho^{2}$ arises from $\left\vert y^{\prime
}\right\vert <\rho$, and that the factor $\rho^{-\frac{4}{q}}Q_{\rho
R}^{\left(  q\right)  }$ arises from parabolic rescaling. These features
remain in play for an arbitrary quadratic surface of positive Gaussian curvature.

Thus using (\ref{Int est}), first with $\rho=2^{-\lambda^{\prime}}$ and then
with $\rho=2^{-\lambda}$, we obtain%
\[
\left\Vert T_{K}f\right\Vert _{L^{q}\left(  B_{R}\right)  }\lesssim2^{-\left(
2-\frac{4}{q}\right)  \lambda^{\prime}}Q_{2^{-\lambda^{\prime}}R}^{\left(
q\right)  }\text{ and }\left\Vert T_{I}f\right\Vert _{L^{q}\left(
B_{R}\right)  }\lesssim2^{-\left(  2-\frac{4}{q}\right)  \lambda
}Q_{2^{-\lambda}R}^{\left(  q\right)  }\ ,
\]
and together with (\ref{contrib Case 2}), we obtain that the contribution
$\left\Vert \mathbf{1}_{\Gamma_{\lambda}\left(  \text{\textbf{Case 2}}\right)
}Tf\right\Vert _{L^{q}\left(  B_{R}\right)  }$ to the norm $\left\Vert
Tf\right\Vert _{L^{q}\left(  B_{R}\right)  }$ satisfies:
\begin{align}
&  \left\Vert \mathbf{1}_{\Gamma_{\lambda}\left(  \text{\textbf{Case 2}%
}\right)  }Tf\right\Vert _{L^{q}\left(  B_{R}\right)  }\leq C\left(
\#\mathcal{G}_{\lambda^{\prime}}\left[  S\right]  \right)  ^{\frac{1}{q}%
}\left(  2^{-\lambda^{\prime}}\right)  ^{2-\frac{4}{q}}Q_{2^{-\lambda^{\prime
}}R}^{\left(  q\right)  }\label{case 2 est}\\
&  \ \ \ \ \ \ \ \ \ \ \ \ \ \ \ +C2^{-\left(  \delta-2+\frac{3}{q}\right)
\lambda}\left(  \#\mathcal{G}_{\lambda}\left[  S\right]  \right)  ^{\frac
{1}{q}}\left(  2^{-\lambda}\right)  ^{2-\frac{4}{q}}Q_{2^{-\lambda}R}^{\left(
q\right)  }\nonumber\\
&  =C2^{\left(  \frac{6}{q}-2\right)  \lambda^{\prime}}Q_{2^{-\lambda^{\prime
}}R}^{\left(  q\right)  }+C2^{\left(  \frac{3}{q}-\delta\right)  \lambda
}Q_{2^{-\lambda}R}^{\left(  q\right)  }\ ,\nonumber
\end{align}
since%
\[
\left(  \#\mathcal{G}_{\lambda^{\prime}}\left[  S\right]  \right)  ^{\frac
{1}{q}}\left(  2^{-\lambda^{\prime}}\right)  ^{2}\left(  2^{-\lambda^{\prime}%
}\right)  ^{-\frac{4}{q}}=2^{\frac{2}{q}\lambda^{\prime}}2^{-2\lambda^{\prime
}}2^{\frac{4}{q}\lambda^{\prime}}=2^{\left(  \frac{6}{q}-2\right)
\lambda^{\prime}},
\]
and%
\[
2^{-\left(  \delta-2+\frac{3}{q}\right)  \lambda}\left(  \#\mathcal{G}%
_{\lambda}\left[  S\right]  \right)  ^{\frac{1}{q}}\left(  2^{-\lambda
}\right)  ^{2-\frac{4}{q}}=2^{-\frac{3}{q}\lambda}2^{\left(  2-\delta\right)
\lambda}2^{\frac{2}{q}\lambda}2^{-2\lambda}2^{\frac{4}{q}\lambda}=2^{\left(
\frac{3}{q}-\delta\right)  \lambda}\ .
\]

\end{proof}

\subsection{Case 3: Dipole interaction\label{Sub Case 3}}

\begin{proof}
[Proof continued]In \textbf{Case 3} we assume the negation of both
\textbf{Case 1} and \textbf{Case 2}. The failure of clustered interaction
implies that there also exists $I_{\ast\ast}^{a}$ with $w_{I_{\ast\ast}^{a}%
}>2^{-\delta\lambda}w_{\ast}^{a}$ and $\left\vert c_{I_{\ast\ast}^{a}%
}-c_{I_{\ast}^{a}}\right\vert >2^{-\gamma\lambda^{\prime}}$, i.e.%
\begin{align*}
\int_{\mathbb{R}^{3}}\left\vert T_{I_{\ast\ast}^{a}}f\left(  z\right)
\right\vert \zeta_{\lambda}\left(  z-a\right)  dz &  =w_{I_{\ast\ast}^{a}%
}>2^{-\delta\lambda}w_{\ast}^{a}=2^{-\delta\lambda}\int_{\mathbb{R}^{3}%
}\left\vert T_{I_{\ast}^{a}}f\left(  z\right)  \right\vert \zeta_{\lambda
}\left(  z-a\right)  dz\ ,\\
&  \mathop{\rm dist}\left(  I_{\ast\ast},I_{\ast}\right)  >2^{-\gamma
\lambda^{\prime}}\ ,
\end{align*}
The simultaneous failure of separated interaction further implies that%
\begin{align}
w_{I}^{a} &  \leq2^{-\alpha\lambda}w_{\ast}^{a}\text{ if }%
\mathop{\rm dist}\left(  c_{I},I_{\ast}^{a}\cup I_{\ast\ast}^{a}\right)
>2^{10}2^{-\beta\lambda},\label{simul}\\
&  \fbox{$%
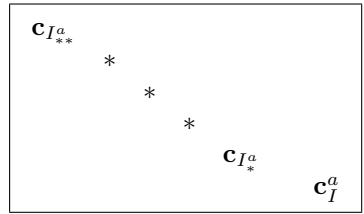
\begin{array}
[c]{ccccccc}%
\mathbf{c}_{I_{\ast\ast}^{a}} &  &  &  &  &  & \\
& \ast &  &  &  &  & \\
&  & \ast &  &  &  & \\
&  &  & \ast &  &  & \\
&  &  &  & \mathbf{c}_{I_{\ast}^{a}} &  & \\
&  &  &  &  &  & \mathbf{c}_{I}^{a}%
\end{array}
$}\nonumber
\end{align}

In this case we will again use parabolic rescaling since the squares $I$ with
near maximal weight, i.e. $2^{-\alpha\lambda}w_{\ast}^{a}$, are clustered
within distance $2^{10}2^{-\beta\lambda}$ of the squares $I_{\ast}^{a}$ and
$I_{\ast\ast}^{a}$. Indeed, arguing as in (\ref{contrib Case 2}) and
(\ref{case 2 est}) above, we then have%
\begin{align*}
&  \left\Vert \mathbf{1}_{\Gamma_{\lambda}\left(  \text{\textbf{Case 3}%
}\right)  }Tf\right\Vert _{L^{q}\left(  B_{R}\right)  }\\
&  \lesssim\left\Vert \mathbf{1}_{\Gamma_{\lambda}\left(  \text{\textbf{Case
3}}\right)  }\sum_{\substack{I:\ \mathop{\rm dist}\left(  c_{I},I_{\ast}%
^{a}\right)  \leq2^{10}2^{-\lambda^{\prime}} \\\text{or }%
\mathop{\rm dist}\left(  c_{I},I_{\ast}^{a}\cup I_{\ast\ast}^{a}\right)
>2^{10}2^{-\lambda^{\prime}}}}T_{I}f\right\Vert _{L^{q}\left(  B_{R}\right)
}+\left\Vert \mathbf{1}_{\Gamma_{\lambda}\left(  \text{\textbf{Case 3}%
}\right)  }\sum_{I:\ \mathop{\rm dist}\left(  c_{I},I_{\ast\ast}^{a}\right)
\leq2^{10}2^{-\lambda^{\prime}}}T_{I}f\right\Vert _{L^{q}\left(  B_{R}\right)
}\\
&  \lesssim C2^{\left(  \frac{6}{q}-2\right)  \lambda^{\prime}}Q_{2^{-\lambda
^{\prime}}R}^{\left(  q\right)  }+C2^{\left(  \frac{3}{q}-\delta\right)
\lambda}Q_{2^{-\lambda}R}^{\left(  q\right)  }\ .
\end{align*}
Note that we needed only to consider the squares satisfying $\mathop{\rm
dist}\left(  c_{I},I_{\ast}^{a}\cup I_{\ast\ast}^{a}\right)  >2^{10}%
2^{-\beta\lambda}$ in (\ref{contrib Case 2}), for just one of the dipoles
$I_{\ast}^{a}$ or $I_{\ast\ast}^{a}$.
\end{proof}

\subsection{Completing the proof}

\begin{proof}
[Proof continued]So far we have shown that if $\lambda^{\prime}=\frac{\beta
}{\gamma}\lambda$ and%
\[
0<\beta<\gamma\ \text{and }0<\alpha=\delta\leq3,
\]
then
\begin{align*}
\left\Vert Tf\right\Vert _{L^{q}\left(  B_{R}\right)  } &  \leq\left\Vert
\mathbf{1}_{\Gamma_{\lambda}\left(  \text{\textbf{Case 1}}\right)
}Tf\right\Vert _{L^{q}\left(  B_{R}\right)  }+\left\Vert \mathbf{1}%
_{\Gamma_{\lambda}\left(  \text{\textbf{Case 2}}\right)  }Tf\right\Vert
_{L^{q}\left(  B_{R}\right)  }+\left\Vert \mathbf{1}_{\Gamma_{\lambda}\left(
\text{\textbf{Case 3}}\right)  }Tf\right\Vert _{L^{q}\left(  B_{R}\right)  }\\
&  \lesssim C_{\varepsilon,\nu,q}2^{\left(  \frac{6}{q}+2+\alpha\right)
\lambda}R^{\varepsilon}+C2^{\left(  \frac{6}{q}-2\right)  \lambda^{\prime}%
}Q_{2^{-\lambda^{\prime}}R}^{\left(  q\right)  }+C2^{\left(  \frac{3}%
{q}-\delta\right)  \lambda}Q_{2^{-\lambda}R}^{\left(  q\right)  }+\left(
2^{-\beta\lambda}\right)  ^{2-\frac{4}{q}}Q_{2^{-\beta\lambda}R}^{\left(
q\right)  }\\
&  \lesssim C_{\varepsilon,\nu,q}2^{\left(  \frac{6}{q}+2+\alpha\right)
\lambda}R^{\varepsilon}+\left[  2^{-\frac{2}{q}\left(  q-3\right)
\lambda^{\prime}}+2^{-\left(  \delta-\frac{3}{q}\right)  \lambda}%
+2^{-\beta\left(  2-\frac{4}{q}\right)  \lambda}\right]  Q_{R}^{\left(
q\right)  }\\
&  \lesssim C_{\varepsilon,\nu,q}2^{\left(  \frac{6}{q}+2+\alpha\right)
\lambda}R^{\varepsilon}+\frac{1}{2}Q_{R}^{\left(  q\right)  }\ ,
\end{align*}
if $q>3$ and both $\lambda$ and $\lambda^{\prime}$ are sufficiently large,
namely%
\begin{equation}
2^{-\left(  q-3\right)  \frac{2}{q}\frac{\beta}{\gamma}\lambda}+2^{-\left(
\delta-\frac{3}{q}\right)  \lambda}+2^{-\beta\left(  2-\frac{4}{q}\right)
\lambda}<\frac{1}{2}.\label{namely}%
\end{equation}
If we take $\beta=1$ and $\alpha=\delta=\gamma=2$, then (\ref{namely})
becomes,%
\[
2^{-\left(  1-\frac{3}{q}\right)  \lambda}+2^{-\left(  2-\frac{3}{q}\right)
\lambda}+2^{-\left(  2-\frac{4}{q}\right)  \lambda}<\frac{1}{2},
\]
which is satisfied if
\[
\left(  1-\frac{3}{q}\right)  \lambda,\left(  2-\frac{4}{q}\right)
\lambda\geq3\text{, in particular if }\lambda=\frac{3q}{q-3}.
\]
Thus
\[
Q_{R}^{\left(  q\right)  }=\sup_{\left\Vert f\right\Vert _{L^{\infty}}\leq
1}\left\Vert Tf\right\Vert _{L^{q}\left(  B_{R}\right)  }\lesssim
C_{\varepsilon,\nu,q}2^{\left(  \frac{6}{q}+4\right)  \lambda}R^{\varepsilon
}+\frac{1}{2}Q_{R}^{\left(  q\right)  },
\]
and absorption now yields the inequality,%
\[
Q_{R}^{\left(  q\right)  }\leq2C_{\varepsilon,\nu,q}2^{\left(  \frac{6}%
{q}+4\right)  \lambda}R^{\varepsilon}\leq2C_{\varepsilon,\nu,q}2^{\left(
\frac{6}{q}+4\right)  \frac{3q}{q-3}}R^{\varepsilon}\leq2C_{\varepsilon,\nu
,q}2^{18\frac{q}{q-3}}R^{\varepsilon},\ \ \ \ \ \text{for all }R\geq1.
\]

Thus the disjoint constant $\nu=2^{10}2^{-\beta\lambda}$ is given by
\begin{equation}
\nu=2^{10}2^{-\frac{3q}{q-3}},\label{nu}%
\end{equation}
which depends only on how much larger $q$ is than $3$. This completes the
proof of Theorem \ref{Loc lin}.
\end{proof}

\section{A smooth Alpert characterization}

First we recall the construction from \cite{Saw7} of smooth Alpert projections
$\left\{  \bigtriangleup_{Q;\kappa}\right\}  _{Q\in\mathcal{D}}$ and
corresponding wavelets $\left\{  h_{Q;\kappa}^{a}\right\}  _{Q\in
\mathcal{D},\ a\in\Gamma_{n}}$ of order $\kappa$ in $n$-dimensional space
$\mathbb{R}^{n}$, giving fairly complete definitions and statements. In fact,
$\left\{  h_{Q;\kappa}^{a}\right\}  _{a\in\Gamma}$ is an orthonormal basis for
the finite dimensional vector subspace of $L^{2}$ that consists of linear
combinations of the indicators of\ the children $\mathfrak{C}\left(  Q\right)
$ of $Q$ multiplied by polynomials of degree at most $\kappa-1$, and such that
the linear combinations have vanishing moments on the cube $Q$ up to order
$\kappa-1$:%
\[
L_{Q;k}^{2}\left(  \mu\right)  \equiv\left\{
f=\mathop{\displaystyle \sum }\limits_{Q^{\prime}\in\mathfrak{C}\left(
Q\right)  }\mathbf{1}_{Q^{\prime}}p_{Q^{\prime};k}\left(  x\right)  :\int
_{Q}f\left(  x\right)  x_{i}^{\ell}d\mu\left(  x\right)  =0,\ \ \ \text{for
}0\leq\ell\leq k-1\text{ and }1\leq i\leq n\right\}  ,
\]
where $p_{Q^{\prime};k}\left(  x\right)  =\sum_{\alpha\in\mathbb{Z}_{+}%
^{n}:\left\vert \alpha\right\vert \leq k-1\ }a_{Q^{\prime};\alpha}x^{\alpha}$
is a polynomial in $\mathbb{R}^{n}$ of degree $\left\vert \alpha\right\vert
=\alpha_{1}+...+\alpha_{n}$ at most $\kappa-1$, and $x^{\alpha}=x_{1}%
^{\alpha_{1}}x_{2}^{\alpha_{2}}...x_{n-1}^{\alpha_{n-1}}$. Let $d_{Q;\kappa
}\equiv\dim L_{Q;\kappa}^{2}\left(  \mu\right)  $ be the dimension of the
finite dimensional linear space $L_{Q;\kappa}^{2}\left(  \mu\right)  $.
Moreover, for each $a\in\Gamma_{n}$, we may assume the wavelet $h_{Q;\kappa
}^{a}$ is a translation and dilation of the unit wavelet $h_{Q_{0};\kappa}%
^{a}$, where $Q_{0}=\left[  0,1\right)  ^{n}$ is the unit cube in
$\mathbb{R}^{n}$.

Given a small positive constant $\eta>0$, define a smooth approximate identity
by $\phi_{\eta}\left(  x\right)  \equiv\eta^{-n}\phi\left(  \frac{x}{\eta
}\right)  $ where $\phi\in C_{c}^{\infty}\left(  B_{\mathbb{R}^{n}}\left(
0,1\right)  \right)  $ has unit integral, $\int_{\mathbb{R}^{n}}\phi\left(
x\right)  dx=1$, and vanishing moments of \emph{positive} order less than
$\kappa$, i.e.
\begin{equation}
\int\phi\left(  x\right)  x^{\gamma}dx=\delta_{\left\vert \gamma\right\vert
}^{0}=\left\{
\begin{array}
[c]{ccc}%
1 & \text{ if } & \left\vert \gamma\right\vert =0\\
0 & \text{ if } & 0<\left\vert \gamma\right\vert <\kappa
\end{array}
\right.  .\label{van pos}%
\end{equation}
The \emph{smooth} Alpert `wavelets' are defined by%
\[
h_{Q;\kappa}^{a,\eta}\equiv h_{Q;\kappa}^{a}\ast\phi_{\eta\ell\left(
Q\right)  },
\]
and we have for $0\leq\left\vert \beta\right\vert <\kappa$,%
\begin{align*}
&  \int h_{Q;\kappa}^{a,\eta}\left(  x\right)  x^{\beta}dx=\int\phi_{\eta
\ell\left(  I\right)  }\ast h_{Q;\kappa}^{a}\left(  x\right)  x^{\beta}%
dx=\int\int\phi_{\eta\ell\left(  I\right)  }\left(  y\right)  h_{Q;\kappa}%
^{a}\left(  x-y\right)  x^{\beta}dx\\
&  =\int\phi_{\eta\ell\left(  I\right)  }\left(  y\right)  \left\{  \int
h_{Q;\kappa}^{a}\left(  x-y\right)  x^{\beta}dx\right\}  dy=\int\phi_{\eta
\ell\left(  I\right)  }\left(  y\right)  \left\{  \int h_{Q;\kappa}^{a}\left(
x\right)  \left(  x+y\right)  ^{\beta}dx\right\}  dy\\
&  =\int\phi_{\eta\ell\left(  I\right)  }\left(  y\right)  \left\{  0\right\}
dy=0,
\end{align*}
by translation invariance of Lebesgue measure.

There is a linear map $S_{\eta}^{\mathcal{D}}=S_{\kappa,\eta}^{\mathcal{D}}$,
bounded and invertible on all $L^{p}\left(  \mathbb{R}^{2}\right)  $ spaces,
$1<p<\infty$,$\,$such that if we define%
\[
\bigtriangleup_{I;\kappa}^{\eta}f\equiv\left(  \bigtriangleup_{I;\kappa
}f\right)  \ast\phi_{\eta\ell\left(  I\right)  },
\]
then%
\[
\bigtriangleup_{I;\kappa}^{\eta}f\equiv\sum_{a\in\Gamma_{n}}\left\langle
\left(  S_{\eta}^{\mathcal{D}}\right)  ^{-1}f,h_{I;\kappa}^{a}\right\rangle
h_{I;\kappa}^{a,\eta}=\sum_{a\in\Gamma_{n}}\left\langle \left(  S_{\eta
}^{\mathcal{D}}\right)  ^{-1}f,h_{I;\kappa}^{a}\right\rangle S_{\eta
}^{\mathcal{D}}h_{I;\kappa}^{a}=\sum_{a\in\Gamma_{n}}\left(  S_{\eta
}^{\mathcal{D}}\bigtriangleup_{I;\kappa}\left(  S_{\eta}^{\mathcal{D}}\right)
^{-1}\right)  f=\sum_{a\in\Gamma_{n}}\bigtriangleup_{I;\kappa}^{\spadesuit
}f\ ,
\]
where $A^{\spadesuit}$ denotes the commutator $S_{\eta}^{\mathcal{D}}A\left(
S_{\eta}^{\mathcal{D}}\right)  ^{-1}$ of an operator $A$ with $S_{\eta
}^{\mathcal{D}}$.

\begin{theorem}
[\cite{Saw7}]\label{reproducing}Let $n\geq2$ and $\kappa\in\mathbb{N}$ with
$\kappa>\frac{n}{2}$. Then there is $\eta_{0}>0$ depending on $n$ and $\kappa
$\ such that for all $0<\eta<\eta_{0}$, and for all grids $\mathcal{D}$ in
$\mathbb{R}^{n} $, and all $1<p<\infty$, there is a bounded invertible
operator $S_{\eta}^{\mathcal{D}}=S_{\kappa,\eta}^{\mathcal{D}}$ on $L^{p}$,
and a positive constant $C_{p,n,\eta}$ such that the collection of functions
$\left\{  h_{I;\kappa}^{a,\eta}\right\}  _{I\in\mathcal{D},\ a\in\Gamma_{n}}$
is a $C_{p,n,\eta}$-frame for $L^{p}$, by which we mean,%
\begin{align}
f\left(  x\right)   &  =\sum_{I\in\mathcal{D},\ a\in\Gamma_{n}}\bigtriangleup
_{I;\kappa}^{\eta}f\left(  x\right)  ,\ \ \ \ \ \text{for all }f\in
L^{p},\label{bounded below}\\
\text{where }\bigtriangleup_{I;\kappa}^{\eta}f  &  \equiv\sum_{a\in\Gamma_{n}%
}\left\langle \left(  S_{\eta}^{\mathcal{D}}\right)  ^{-1}f,h_{I;\kappa}%
^{a}\right\rangle \ h_{I;\kappa}^{a,\eta}\ ,\nonumber
\end{align}
and with convergence of the sum in both the $L^{p}$ norm and almost
everywhere, and%
\begin{equation}
\frac{1}{C_{p,n,\eta}}\left\Vert f\right\Vert _{L^{p}}\leq\left\Vert \left(
\sum_{I\in\mathcal{D}}\left\vert \bigtriangleup_{I;\kappa}^{\eta}f\right\vert
^{2}\right)  ^{\frac{1}{2}}\right\Vert _{L^{p}}\leq C_{p,n,\eta}\left\Vert
f\right\Vert _{L^{p}},\ \ \ \ \ \text{for all }f\in L^{p}.\label{square est}%
\end{equation}
Moreover, the smooth Alpert wavelets $\left\{  h_{I;\kappa}^{a,\eta}\right\}
_{I\in\mathcal{D},\ a\in\Gamma_{n}}$ are translation and dilation invariant in
the sense that $h_{I;\kappa}^{a,\eta}$ is a translate and dilate of the mother
Alpert wavelet $h_{I_{0};\kappa}^{a,\eta}$ where $I_{0}$ is the unit cube in
$\mathbb{R}^{n}$.
\end{theorem}

\begin{notation}
\label{Notation Alpert} We will often drop the index $a$ parameterized by the
finite set $\Gamma_{n}$ as it plays no essential role in most of what follows,
and it will be understood that when we write
\[
\bigtriangleup_{Q;\kappa}^{\eta}f=\left\langle \left(  S_{\eta}^{\mathcal{D}%
}\right)  ^{-1}f,h_{Q;\kappa}\right\rangle h_{Q;\kappa}^{\eta}=\widehat
{f}\left(  Q\right)  h_{Q;\kappa}^{\eta},
\]
we \emph{actually} mean the Alpert \emph{pseudoprojection},%
\[
\bigtriangleup_{Q;\kappa}^{\eta}f=\sum_{a\in\Gamma_{n}}\left\langle \left(
S_{\eta}^{\mathcal{D}}\right)  ^{-1}f,h_{Q;\kappa}^{a}\right\rangle
h_{Q;\kappa}^{\eta,a}=\sum_{a\in\Gamma_{n}}\widehat{f_{a}}\left(  Q\right)
h_{Q;\kappa}^{a,\eta}\ ,
\]
where $\widehat{f_{a}}\left(  Q\right)  $ is a convenient abbreviation for the
inner product $\left\langle \left(  S_{\eta}^{\mathcal{D}}\right)
^{-1}f,h_{Q;\kappa}^{a}\right\rangle $ when $\kappa$ is understood. More
precisely, one can view $\widehat{f}\left(  Q\right)  =\left\{  \widehat
{f_{a}}\left(  Q\right)  \right\}  _{a\in\Gamma_{n}}$ and $h_{Q;\kappa}^{\eta
}=\left\{  h_{Q;\kappa}^{a,\eta}\right\}  _{a\in\Gamma_{n}}$ as sequences of
numbers and functions indexed by $\Gamma_{n}$, in which case $\widehat
{f}\left(  Q\right)  h_{Q;\kappa}^{\eta}$ is the dot product of these two
sequences. No confusion should arise between the Alpert coefficient
$\widehat{g}\left(  Q\right)  $, $Q\in\mathcal{G}\left[  U\right]  $ and the
Fourier transform $\widehat{g}\left(  \xi\right)  $, $\xi\in\mathbb{R}^{3}$,
as the argument in the first is a square in $\mathcal{G}\left[  U\right]  $,
while the argument in the second is a point in $\mathbb{R}^{3}$.
\end{notation}

For $s\in\mathbb{N}$, and $K\in\mathcal{G}\left[  S\right]  $, with
$S\subset\mathbb{R}^{2}$ centered at the origin as above, and with
$\ell\left(  K\right)  \geq2^{-s}$, we define the smooth Alpert
pseudoprojection at scale $s$ by,%
\begin{equation}
\mathsf{Q}_{s,K;\kappa}^{\eta}f\equiv\sum_{I\in\mathcal{G}_{s}\left[
K\right]  }\bigtriangleup_{I;\kappa}^{\eta}f\text{ and }f_{s,K}^{\Phi}%
\equiv\Phi_{\ast}\mathsf{Q}_{s,K;\kappa}^{\eta}f=\sum_{I\in\mathcal{G}%
_{s}\left[  K\right]  }\Phi_{\ast}\bigtriangleup_{I;\kappa}^{\eta
}f.\label{def A proj}%
\end{equation}

\subsection{Initial setup and statement of the main Alpert characterization}

Now we return to three dimensions. We recall some of the notation in
\cite[Subsection 1.4]{Saw7} regarding local coordinates on the sphere, and
pushforwards of smooth Alpert wavelets near the origin in $\mathbb{R}^{2}$.
Fix a small cube $U_{0}$ in $\mathbb{R}^{n-1}$ with side length a negative
power of $2$, and such that there is a translation $\mathcal{G}$ of the
standard grid on $\mathbb{R}^{n-1}$ with the property that $U_{0}%
\in\mathcal{G}$, the grandparent $\pi_{\mathcal{G}}^{\left(  2\right)  }U_{0}$
of $U_{0}$ has the origin as a vertex, and $U_{0}$ is an interior grandchild
of $U\equiv\pi_{\mathcal{G}}^{\left(  2\right)  }U_{0}$, so that%
\begin{equation}
U_{0},U\in\mathcal{G}\text{ with }U_{0}\subset\frac{1}{2}U\text{.}%
\label{support}%
\end{equation}

Then parameterize a patch of the paraboloid $\mathbb{P}^{2}$ in the usual way,
i.e. $\Phi:U\rightarrow\mathbb{S}^{2}$ by
\[
z=\Phi\left(  x\right)  \equiv\left(  x,\left\vert x\right\vert ^{2}\right)
=\left(  x_{1},x_{2},x_{1}^{2}+x_{2}^{2}\right)  .
\]
For $f\in L^{p}\left(  U\right)  $, define
\[
\mathcal{E}f\left(  \xi\right)  =\mathcal{E}_{U}f\left(  \xi\right)
\equiv\mathcal{F}\left(  \Phi_{\ast}\left[  f\left(  x\right)  dx\right]
\right)  =\int_{U}e^{-i\Phi\left(  x\right)  \cdot\xi}f\left(  x\right)  dx,
\]
where $\mathcal{F}$ is the Fourier transform in $\mathbb{R}^{3}$.

Recall that for $\nu>0$, we say that a triple $\left(  U_{1},U_{2}%
,U_{3}\right)  $ of squares in $U\subset B_{\mathbb{R}^{2}}\left(  0,\frac
{1}{2}\right)  $ is $\nu$\emph{-disjoint} if
\[
\mathop{\rm diam}\left[  \Phi\left(  U_{k}\right)  \right]  \approx
\mathop{\rm dist}\left[  \Phi\left(  U_{k}\right)  ,\bigcup_{j:\ j\neq k}%
\Phi\left(  U_{j}\right)  \right]  \geq\nu,\text{ for }1\leq k\leq3,
\]

\begin{definition}
Let $\varepsilon,\nu>0$ and $0<\delta<1$, $\kappa\in\mathbb{N}$ and
$1<q<\infty$. Set $A\left(  0,2^{r}\right)  \equiv B\left(  0,2^{r}\right)
\setminus B\left(  0,2^{r-1}\right)  $ for $r\in\mathbb{N}$. Denote by
$\mathcal{A}_{\mathop{\rm disj}\nu}^{\kappa,\delta}\left(  \otimes
_{3}L^{\infty}\rightarrow L^{\frac{q}{3}};\varepsilon\right)  $ the trilinear
smooth Alpert inequality,%
\begin{equation}
\left(  \int_{A\left(  0,2^{r}\right)  }\left(  \left\vert \mathcal{E}%
\mathsf{Q}_{s_{1},U_{1}}^{\eta}f_{1}\left(  \xi\right)  \right\vert
\ \left\vert \mathcal{E}\mathsf{Q}_{s_{2},U_{2}}^{\eta}f_{2}\left(
\xi\right)  \right\vert \ \left\vert \mathcal{E}\mathsf{Q}_{s_{3},U_{3}}%
^{\eta}f_{3}\left(  \xi\right)  \right\vert \right)  ^{\frac{q}{3}}%
\ d\xi\right)  ^{\frac{3}{q}}\leq C_{\varepsilon,\nu,\kappa,\delta
,q}2^{\varepsilon r}\left\Vert f_{1}\right\Vert _{L^{\infty}}\left\Vert
f_{2}\right\Vert _{L^{\infty}}\left\Vert f_{3}\right\Vert _{L^{\infty}%
}\ ,\label{annular}%
\end{equation}
for all $r\in\mathbb{N}$, all $\nu$\emph{-disjoint} triples $\left(
U_{1},U_{2},U_{3}\right)  $, all smooth Alpert pseudoprojections
$\mathsf{Q}_{s_{k},U_{k}}^{\eta}$ with moment vanishing parameter $\kappa$ and
with $s_{1}\leq s_{2}$ and $\frac{r}{1+\delta}<s_{2}\leq s_{3}<\frac
{r}{1-\delta} $, and all $f_{1},f_{2},f_{3}\in L^{\infty}$.
\end{definition}

\begin{theorem}
\label{main Alpert}Let $0<\delta<1$ and $\kappa>\frac{20}{\delta}$. The
Fourier extension conjecture (\ref{FEC}) for the paraboloid in $\mathbb{R}%
^{3}$ holds \emph{if and only if} for every $q>3$ there is $\nu>0$ depending
only on $q$, such that the disjoint \emph{smooth Alpert }trilinear inequality
$\mathcal{A}_{\mathop{\rm disj}\nu}^{\kappa,\delta}\left(  \otimes
_{3}L^{\infty}\rightarrow L^{\frac{q}{3}};\varepsilon\right)  $ holds for all
$\varepsilon>0$.
\end{theorem}

\begin{description}
\item[Note] Integration on the left hand side of (\ref{annular}) is taken over
the annulus $A\left(  0,2^{s}\right)  $ rather than the ball $B\left(
0,2^{s}\right)  $. Thus one should think of the $\mathcal{A}$ in
$\mathcal{A}_{\mathop{\rm disj}\nu}^{\kappa}\left(  \otimes_{3}L^{p}%
\rightarrow L^{\frac{q}{3}};\varepsilon\right)  $ as standing for `annulus' as
well as `Alpert'.
\end{description}

Note that the inequality (\ref{annular}) lies in the fully resonant spectrum,
since
\[
\mathcal{E}\bigtriangleup_{I_{k};\kappa}^{\eta}f_{k}\left(  \xi\right)
=\sum_{I_{k}\in\mathcal{G}_{s}\left[  U_{k}\right]  }\widehat{f_{k}}\left(
I_{k}\right)  \int_{\mathbb{R}^{2}}e^{-i\Phi\left(  x\right)  \cdot\xi
}h_{I_{k};\kappa}^{\eta}\left(  x\right)  dx\ ,
\]
where the wavelength of the oscillatory factor $e^{-i\Phi\left(  x\right)
\cdot\xi}$ on the support of $h_{I_{k};\kappa}^{\eta}$ is roughly $\frac
{1}{\left\vert \xi\right\vert }\in\left(  2^{-\left(  1+\delta\right)
s},2^{-\left(  1-\delta\right)  s}\right)  $, and the side length of $I_{k}%
\,$is $2^{-s}$.

The smooth Alpert pseudoprojection $\mathsf{Q}_{s,U_{k}}^{\eta}=\sum
_{I\in\mathcal{G}_{s}\left[  U_{k}\right]  }\bigtriangleup_{I;\kappa}^{\eta}%
$at level $s\in\mathbb{N}$ was introduced in \cite{Saw7}, where a
probabilistic analogue of the Fourier extension conjecture was proved, but
unlike the probabilistic analysis in \cite{Saw7}, which included a wide range
of fully resonant inner products (from $2^{s}\,$to $2^{2s}$), the range of
resonance in the disjoint smooth Alpert trilinear inequality $\mathcal{A}%
_{\mathop{\rm
disj}\nu}^{\kappa,\delta}\left(  \otimes_{3}L^{q}\rightarrow L^{\frac{q}{3}%
};\varepsilon\right)  $ is greatly reduced due to the convolution of disjoint
patches on the sphere. It is crucial that we do not need to obtain estimates
on how the constants depend on transversality in the trilinear result,
something that is enabled by the argument\ of Bourgain and Guth \cite{BoGu}.

The disjoint \emph{smooth Alpert }trilinear inequality $\mathcal{A}%
_{\mathop{\rm disj}\nu}^{\kappa,\delta}\left(  \otimes L^{\infty}\rightarrow
L^{\frac{q}{3}};\varepsilon\right)  $ represents the weakest formulation of
the Fourier extension conjecture that the authors could find to date,
requiring only smooth Alpert projections of bounded functions with the two
largest scales near $s$ in a disjoint trilinear inequality, integration over
the corresponding `small' annulus $A\left(  0,2^{s}\right)  $ of resonance,
and permitting the familiar small $\varepsilon$-power growth in $s$.

At the end of the paper we show how the \emph{probabilistic} Fourier extension
theorem in \cite{Saw7} can be proved using a square function modification of
the arguments in this paper, providing a new and arguably simpler proof.

\subsection{Convolution of $\nu$-disjoint singular measures on the paraboloid}

Let $\mu^{1}\equiv\Phi_{\ast}\mathsf{Q}_{s_{1},U_{1}}^{\eta}f_{1}$ and
$\mu^{2}\equiv\Phi_{\ast}\mathsf{Q}_{s_{2},U_{2}}^{\eta}f_{2}$ denote singular
measures on the paraboloid, that are pushforwards of smooth Alpert projections
at levels $s_{1}<s_{2}$ of functions $f_{k}\in L^{p}\left(  U_{k}\right)  $,
$1<p<\infty$, and where $\mathop{\rm diam}\left(  U_{1}\right)  \approx
\mathop{\rm diam}\left(  U_{2}\right)  \approx\mathop{\rm
dist}\left(  U_{1},U_{2}\right)  \gtrsim\nu>0$. For $z\in\mathbb{R}^{3}$,
denote by $\omega_{z}$ the translate of a measure $\omega$ by $z$. We use
duality to compute the convolution $\mu^{1}\ast\mu^{2}$ in terms of the
measure-valued integral $\iiint_{w\in\mathbb{R}^{3}}\left[  \mu_{w}^{1}\left(
\cdot\right)  \right]  d\mu^{2}\left(  w\right)  $ as follows. For $F$ a
continuous function on $\mathbb{R}^{3}$, write
\begin{align*}
&  \left\langle F,\Phi_{\ast}\mathsf{Q}_{s_{1},U_{1}}^{\eta}f_{1}\ast
\Phi_{\ast}\mathsf{Q}_{s_{2},U_{2}}^{\eta}f_{2}\right\rangle =\left\langle
F,\mu^{1}\ast\mu^{2}\right\rangle =\left\langle F\left(  \cdot\right)
,\iiint_{w\in\mathbb{R}^{3}}\mu_{w}^{1}\left(  \cdot\right)  d\mu^{2}\left(
w\right)  \right\rangle \\
&  =\iiint_{z\in\mathbb{R}^{3}}F\left(  z\right)  d\left[  \iiint
_{w\in\mathbb{R}^{3}}\mu_{w}^{1}\left(  z\right)  d\mu^{2}\left(  w\right)
\right]  =\iiint_{w\in\mathbb{R}^{3}}\left\{  \iiint_{z\in\mathbb{R}^{3}%
}F\left(  z\right)  d\mu_{w}^{1}\left(  z\right)  \right\}  d\mu^{2}\left(
w\right) \\
&  =\iiint_{w\in\mathbb{R}^{3}}\left\{  \iiint_{z\in\mathbb{R}^{3}}F\left(
z-w\right)  d\mu^{1}\left(  z\right)  \right\}  d\mu^{2}\left(  w\right)  ,
\end{align*}
and using the definitions of $\mu^{1}$ and $\mu^{2}$ as pushforwards
respectively of $\mathsf{Q}_{s_{1},U_{1}}^{\eta}f_{1}\left(  v\right)  dv$ and
$\mathsf{Q}_{s_{2},U_{2}}^{\eta}f_{2}\left(  u\right)  du$ by $\Phi$, we see
that%
\begin{align*}
\left\langle F,\mu^{1}\ast\mu^{2}\right\rangle  &  =\iiint_{w\in\mathbb{R}%
^{3}}\left\{  \iiint_{z\in\mathbb{R}^{3}}F\left(  z-w\right)  d\mu^{1}\left(
z\right)  \right\}  d\mu^{2}\left(  w\right) \\
&  =\iiint_{w\in\mathbb{R}^{3}}\left\{  \iint_{v\in U_{1}}F\left(  \Phi\left(
v\right)  -w\right)  \mathsf{Q}_{s_{1},U_{1}}^{\eta}f_{1}\left(  v\right)
dv\right\}  d\mu^{2}\left(  w\right) \\
&  =\iint_{u\in U_{2}}\iint_{v\in U_{1}}F\left(  \Phi\left(  v\right)
-\Phi\left(  u\right)  \right)  \mathsf{Q}_{s_{1},U_{1}}^{\eta}f_{1}\left(
v\right)  dv\mathsf{Q}_{s_{2},U_{2}}^{\eta}f_{2}\left(  u\right)  du.
\end{align*}

Taking limits we can let $F=\delta_{a}$, so that for $\left(  v,u\right)  \in
U_{1}\times U_{2}$,%
\begin{align}
\left\langle \delta_{a},\mu^{1}\ast\mu^{2}\right\rangle  &  =\iiiint_{\left(
v,u\right)  \in U_{1}\times U_{2}}\delta_{a}\left(  \Phi\left(  v\right)
-\Phi\left(  u\right)  \right)  \mathsf{Q}_{s_{1},U_{1}}^{\eta}f_{1}\left(
v\right)  dv\mathsf{Q}_{s_{2},U_{2}}^{\eta}f_{2}\left(  u\right)
du\label{beta}\\
&  =\iiiint_{E_{a}}\mathsf{Q}_{s_{1},U_{1}}^{\eta}f_{1}\left(  v\right)
\mathsf{Q}_{s_{2},U_{2}}^{\eta}f_{2}\left(  u\right)  \beta_{a}\left(
u,v\right)  dvdu,\nonumber
\end{align}
where $E_{a}$ is a line in $U_{1}\times U_{2}$ given by%
\begin{align*}
E_{a}  &  \equiv\left\{  \left(  v,u\right)  \in U_{1}\times U_{2}%
:v-u=a^{\prime}\text{ and }\left\vert v\right\vert ^{2}-\left\vert
u\right\vert ^{2}=a_{3}\right\} \\
&  =\left\{  \left(  u+a^{\prime},u\right)  :\left\vert u+a^{\prime
}\right\vert ^{2}-\left\vert u\right\vert ^{2}=a_{3}\right\}  =\left\{
\left(  u+a^{\prime},u\right)  :2a^{\prime}\cdot u=a_{3}-\left\vert a^{\prime
}\right\vert ^{2}\right\} \\
&  =\left\{  \left(  v,v-a^{\prime}\right)  :\left\vert v\right\vert
^{2}-\left\vert v-a^{\prime}\right\vert ^{2}=a_{3}\right\}  =\left\{  \left(
v,v-a^{\prime}\right)  :2a^{\prime}\cdot v=a_{3}+\left\vert a^{\prime
}\right\vert ^{2}\right\}  ,
\end{align*}
and $\beta_{a}\left(  u,v\right)  $ is the smooth density arising from the
limiting passage from $F$ to $\delta_{a}$. This can be seen from an
application of the implicit function theorem using $\left\vert a^{\prime
}\right\vert =\left\vert v-u\right\vert \geq\nu>0$ by the $\nu$-disjoint
assumption (or by direct calculation). We further conclude from the implicit
function theorem that%
\begin{equation}
a\rightarrow\left\langle \delta_{a},\mu^{1}\ast\mu^{2}\right\rangle \text{ is
a smooth function of }a\text{ which is \emph{adapted to scale }}\frac{1}{\nu
}\min\left\{  2^{-s_{1}},2^{-s_{2}}\right\}  =\frac{1}{\nu}2^{-s_{2}%
},\label{ift adapted}%
\end{equation}
where by adapted to scale $\delta>0$, we mean that $m^{th}$ order derivatives
are bounded by $C_{m}\left(  \frac{1}{\delta}\right)  ^{m}$. Moreover, the
density `$\mu^{1}\ast\mu^{2}\left(  a\right)  $' of the absolutely continuous
measure $\mu^{1}\ast\mu^{2}=\left(  \mu^{1}\ast\mu^{2}\right)  \left(
a\right)  da$ is given by%
\begin{equation}
\mu^{1}\ast\mu^{2}\left(  a\right)  =\left\langle \delta_{a},\mu^{1}\ast
\mu^{2}\right\rangle .\label{density}%
\end{equation}

At this point we note the crude inequality%
\begin{align*}
\left\vert \widehat{f}\left(  I\right)  \right\vert  &  =\left\vert
\left\langle \left(  S_{\kappa,\eta}^{\mathcal{G}}\right)  ^{-1}f,h_{I;\kappa
}\right\rangle \right\vert \leq\left\Vert \left(  S_{\kappa,\eta}%
^{\mathcal{G}}\right)  ^{-1}f\right\Vert _{L^{p}}\left\Vert h_{I;\kappa
}\right\Vert _{L^{p^{\prime}}}\\
&  \lesssim\left\Vert \left(  S_{\kappa,\eta}^{\mathcal{G}}\right)
^{-1}\right\Vert _{L^{p}\rightarrow L^{p}}\left\Vert f\right\Vert
_{L^{p}\left(  U\right)  }\left\Vert h_{I;\kappa}\right\Vert _{L^{p}%
}\left\Vert \mathbf{1}_{I}\right\Vert _{L^{p^{\prime}}}\lesssim\left\Vert
f\right\Vert _{L^{p}\left(  U\right)  }\ell\left(  I\right)  ^{-1}\ell\left(
I\right)  ^{\frac{2}{p^{\prime}}}=\ell\left(  I\right)  ^{\frac{1}{p^{\prime}%
}-\frac{1}{p}}\left\Vert f\right\Vert _{L^{p}}\ ,
\end{align*}
for any $1<p<\infty$, which gives%
\begin{equation}
\left\vert \bigtriangleup_{I;\kappa}^{\eta}f\right\vert =\left\vert
\widehat{f}\left(  I\right)  h_{I;\kappa}^{\eta}\right\vert \lesssim
\ell\left(  I\right)  ^{-\left(  \frac{1}{p}-\frac{1}{p^{\prime}}\right)
}\left\Vert f\right\Vert _{L^{p}}\ell\left(  I\right)  ^{-1}\mathbf{1}%
_{\left(  1+\eta\ell\left(  I\right)  \right)  I}=\ell\left(  I\right)
^{-\frac{2}{p}}\left\Vert f\right\Vert _{L^{p}}\mathbf{1}_{\left(  1+\eta
\ell\left(  I\right)  \right)  I},\ \ \ \ \ 1<p<\infty.\label{precrude}%
\end{equation}

Altogether (\ref{ift adapted}), (\ref{density}) and (\ref{precrude}) prove the
following lemma.

\begin{lemma}
\label{scales}For $\mu^{1}$ and $\mu^{2}$ as above, i.e. $\mu^{1}\equiv
\Phi_{\ast}\mathsf{Q}_{s_{1},U_{1}}^{\eta}f_{1}$ and $\mu^{2}\equiv\Phi_{\ast
}\mathsf{Q}_{s_{2},U_{2}}^{\eta}f_{2}$ where $s_{1}<s_{2}$ and $f_{k}\in
L^{p}\left(  U_{k}\right)  $, $1<p<\infty$, and $U_{1}$ and $U_{2}$ are $\nu
$-separated. Then the following derivative estimates hold,%
\begin{align*}
\left\vert \nabla_{a}^{m}\left(  \mu^{1}\ast\mu^{2}\right)  \left(  a\right)
\right\vert  &  \lesssim C_{m}\frac{1}{\nu^{m}}2^{ms_{2}}2^{\frac{2}{p}\left(
s_{1}+s_{2}\right)  }\left\Vert f_{1}\right\Vert _{L^{p}\left(  U_{1}\right)
}\left\Vert f_{2}\right\Vert _{L^{p}\left(  U_{2}\right)  },\\
\text{for all }m  &  \geq0\text{, }1<p<\infty\text{, and }a\in
\mathop{\rm Supp}\left(  \mu^{1}\ast\mu^{2}\right)  .
\end{align*}

\end{lemma}

Using Lemma \ref{scales} for $\nu$-disjoint singular measures on the
paraboloid, we can now prove Theorem \ref{main Alpert}. Indeed, we will
exploit the smoothness and moment vanishing properties of smooth Alpert
wavelets to obtain geometric decay in\ convolutions of singular measures
supported on the paraboloid.

\begin{proof}
[Proof of Theorem \ref{main Alpert}]Let $0<\delta<1$ and $\kappa>\frac
{20}{\delta}$ as in the statement of Theorem \ref{main Alpert}. Without loss
of generality we assume that $s_{1}<s_{2}<s_{3}$ and $R=2^{r}$, and we will
repeatedly use the inequalities (\ref{ift adapted}), (\ref{density}) and
(\ref{precrude}) to prove Theorem \ref{main Alpert} in three cases, followed
by a wrapup.
\end{proof}

\subsection{Case 1: $s_{3}$ is large}

\begin{proof}
[Proof continued]First suppose that $s_{3}>\frac{r}{1-\delta}$. Then for
$\left\vert \xi\right\vert \approx2^{r}$, we have upon using $\kappa$-moment
vanishing of $\bigtriangleup_{I_{3};\kappa}^{\eta}f_{3}$ and differentiating
\[
\exp_{\Phi,\xi}\left(  x\right)  \equiv e^{-i\Phi\left(  x\right)  \cdot\xi
}=e^{-i\Phi\left(  c_{I_{3}}\right)  \cdot\xi}e^{-i\left[  \Phi\left(
x\right)  -\Phi\left(  c_{I_{3}}\right)  \right]  \cdot\xi}%
\]
with respect to $x$, together with (\ref{precrude}),
\begin{align*}
&  \left\vert \mathcal{E}\mathsf{Q}_{s_{3},U_{3}}^{\eta}f_{3}\left(
\xi\right)  \right\vert =\left\vert \left[  \Phi_{\ast}\bigtriangleup
_{I_{3};\kappa}^{\eta}f_{3}\right]  ^{\wedge}\left(  \xi\right)  \right\vert
=\left\vert \int e^{-iz\cdot\xi}\Phi_{\ast}\bigtriangleup_{I_{3};\kappa}%
^{\eta}f_{3}\left(  z\right)  dz\right\vert =\left\vert \int e^{-i\Phi\left(
x\right)  \cdot\xi}\bigtriangleup_{I_{3};\kappa}^{\eta}f_{3}\left(  x\right)
dx\right\vert \\
&  =\left\vert \int\left\{  e^{-i\Phi\left(  x\right)  \cdot\xi}-\sum
_{k=1}^{\kappa-1}\frac{\left[  \left(  x-c_{I_{3}}\right)  \cdot\nabla\right]
^{k}}{k!}\exp_{\Phi,\xi}\left(  c_{I_{3}}\right)  \right\}  \bigtriangleup
_{I_{3};\kappa}^{\eta}f_{3}\left(  x\right)  dx\right\vert \\
&  =\left\vert \int\left\{  \frac{\left[  \left(  x-c_{I_{3}}\right)
\cdot\nabla\right]  ^{\kappa}}{\kappa!}\exp_{\Phi,\xi}\left(  \theta_{I_{3}%
}\right)  \right\}  \bigtriangleup_{I_{3};\kappa}^{\eta}f_{3}\left(  x\right)
dx\right\vert \\
&  \lesssim C_{\kappa}\int\left\vert 2^{-s_{3}}\xi\right\vert ^{\kappa
}\left\vert \bigtriangleup_{I_{3};\kappa}^{\eta}f_{3}\left(  x\right)
\right\vert dx\lesssim C_{\kappa}\int\left(  2^{-s_{3}}\left\vert
\xi\right\vert \right)  ^{\kappa}2^{s_{3}\frac{2}{p}}\left\Vert f_{3}%
\right\Vert _{L^{p}}\mathbf{1}_{\left(  1+\eta\ell\left(  I_{3}\right)
\right)  I_{3}}dx\\
&  \leq C_{\kappa}2^{-\left(  \kappa-\frac{2}{p}\right)  s_{3}}2^{\kappa
r}\left\Vert f_{3}\right\Vert _{L^{p}}=C_{\kappa}2^{-\kappa\left(
s_{3}-r\right)  }2^{\frac{2}{p}s_{3}}\left\Vert f_{3}\right\Vert _{L^{p}},
\end{align*}
and so%
\begin{align*}
&  \left(  \int_{A\left(  0,2^{r}\right)  }\left(  \left\vert \mathcal{E}%
\mathsf{Q}_{s_{1},U_{1}}^{\eta}f_{1}\left(  \xi\right)  \right\vert
\ \left\vert \mathcal{E}\mathsf{Q}_{s_{2},U_{2}}^{\eta}f_{2}\left(
\xi\right)  \right\vert \ \left\vert \mathcal{E}\mathsf{Q}_{s_{3},U_{3}}%
^{\eta}f_{3}\left(  \xi\right)  \right\vert \right)  ^{\frac{q}{3}}%
\ d\xi\right)  ^{\frac{3}{q}}\\
&  \lesssim\left(  \int_{A\left(  0,2^{r}\right)  }\left(  \left\Vert
f_{1}\right\Vert _{L^{1}}\left\Vert f_{2}\right\Vert _{L^{1}}C_{\kappa
}2^{-\kappa\left(  s_{3}-r\right)  }2^{\frac{2}{p}s_{3}}\left\Vert
f_{3}\right\Vert _{L^{p}}\right)  ^{\frac{q}{3}}\ d\xi\right)  ^{\frac{3}{q}%
}\\
&  \lesssim C_{\kappa}2^{-\kappa\left(  s_{3}-r\right)  }2^{\frac{2}{p}s_{3}%
}\left(  \int_{A\left(  0,2^{r}\right)  }d\xi\right)  ^{\frac{3}{q}}\left\Vert
f_{1}\right\Vert _{L^{1}}\left\Vert f_{2}\right\Vert _{L^{1}}\left\Vert
f_{3}\right\Vert _{L^{p}}\\
&  \lesssim C_{\kappa}2^{-\kappa\left(  s_{3}-r\right)  }2^{\frac{2}{p}s_{3}%
}2^{\frac{9}{q}r}\left\Vert f_{1}\right\Vert _{L^{p}}\left\Vert f_{2}%
\right\Vert _{L^{p}}\left\Vert f_{3}\right\Vert _{L^{p}}\\
&  \leq C_{\kappa}2^{-\kappa\left(  s_{3}-\left(  1-\delta\right)
s_{3}\right)  }2^{\frac{2}{p}s_{3}}2^{\frac{9}{q}\left(  1-\delta\right)
s_{3}}\left\Vert f_{1}\right\Vert _{L^{p}}\left\Vert f_{2}\right\Vert _{L^{p}%
}\left\Vert f_{3}\right\Vert _{L^{p}}\\
&  =C_{\kappa}2^{-\kappa\delta s_{3}}2^{\left(  \frac{9}{q}\left(
1-\delta\right)  +\frac{2}{p}\right)  s_{3}}\left\Vert f_{1}\right\Vert
_{L^{p}}\left\Vert f_{2}\right\Vert _{L^{p}}\left\Vert f_{3}\right\Vert
_{L^{p}}\ .
\end{align*}

Since $\kappa>\frac{20}{\delta}$, $q>3$ and $p>1$, we have $\kappa
\delta-\left(  \frac{9}{q}\left(  1-\delta\right)  +\frac{2}{p}\right)
\geq15>1$. Then summing in $s_{3}>\frac{r}{1-\delta}$ and using $q>3$ in
Minkowski's inequality gives
\begin{align}
&  \left(  \int_{A\left(  0,2^{r}\right)  }\left(  \sum_{s_{1}\leq s_{2}\leq
s_{3}\text{ and }s_{3}>\frac{r}{1-\delta}}\left\vert \mathcal{E}%
\mathsf{Q}_{s_{1},U_{1}}^{\eta}f_{1}\left(  \xi\right)  \right\vert
\ \left\vert \mathcal{E}\mathsf{Q}_{s_{2},U_{2}}^{\eta}f_{2}\left(
\xi\right)  \right\vert \ \left\vert \mathcal{E}\mathsf{Q}_{s_{3},U_{3}}%
^{\eta}f_{3}\left(  \xi\right)  \right\vert \right)  ^{\frac{q}{3}}%
\ d\xi\right)  ^{\frac{3}{q}}\label{first half}\\
&  \lesssim\sum_{s_{1}\leq s_{2}\leq s_{3}\text{ and }s_{3}>\frac{r}{1-\delta
}}\left(  \int_{A\left(  0,2^{r}\right)  }\left(  \left\vert \mathcal{E}%
\mathsf{Q}_{s_{1},U_{1}}^{\eta}f_{1}\left(  \xi\right)  \right\vert
\ \left\vert \mathcal{E}\mathsf{Q}_{s_{2},U_{2}}^{\eta}f_{2}\left(
\xi\right)  \right\vert \ \left\vert \mathcal{E}\mathsf{Q}_{s_{3},U_{3}}%
^{\eta}f_{3}\left(  \xi\right)  \right\vert \right)  ^{\frac{q}{3}}%
\ d\xi\right)  ^{\frac{3}{q}}\nonumber\\
&  \lesssim C_{\kappa}\sum_{s_{3}>\frac{r}{1-\delta}}s_{3}^{2}2^{-s_{3}%
}\left\Vert f_{1}\right\Vert _{L^{p}}\left\Vert f_{2}\right\Vert _{L^{p}%
}\left\Vert f_{3}\right\Vert _{L^{p}}\lesssim C_{\kappa}r^{2}2^{-\frac
{r}{1-\delta}}\left\Vert f_{1}\right\Vert _{L^{p}}\left\Vert f_{2}\right\Vert
_{L^{p}}\left\Vert f_{3}\right\Vert _{L^{p}}.\nonumber
\end{align}

\end{proof}

\subsection{Case 2: $s_{2}$ is small}

\begin{proof}
[Proof continued]Next we suppose $s_{1}<s_{2}\leq\frac{r}{1+\delta}$ and
$s_{3}\leq\frac{r}{1-\delta}$, and $I_{1}\in\mathcal{G}_{s_{1}}\left[
U_{1}\right]  $ and $I_{2}\in\mathcal{G}_{s_{2}}\left[  U_{2}\right]  $. Then
from Lemma \ref{scales}, we obtain that%
\[
F_{I_{1},I_{2}}\equiv\Phi_{\ast}\bigtriangleup_{I_{1};\kappa}^{\eta}f\ast
\Phi_{\ast}\bigtriangleup_{I_{2};\kappa}^{\eta}f\left(  z\right)
\]
is compactly supported in the three dimensional rectangle $2\left(
\Phi\left(  I_{1}\right)  +\Phi\left(  I_{2}\right)  \right)  $, and smoothly
adapted to scale $\nu2^{-s_{2}}$. Note that the smoothness scale from Lemma
\ref{scales} is better than that obtained from the usual localization
$\mathbf{1}_{A\left(  0,2^{r}\right)  }\leq\widehat{\varphi_{2^{-r}}}$, which
is just $2^{-r}$\footnote{This smoothness $2^{-r}$ arising from localization
is what was used in \cite{Saw7}, and the sharper smoothness $\nu2^{-s_{2}}$
arising from $\nu$-disjointness used here is key to our alternate approach to
the probabilistic Fourier extension theorem below.}. Thus for $\xi\in A\left(
0,2^{r}\right)  $, we obtain from (\ref{precrude}) that%
\begin{align*}
&  \left\vert \left[  \Phi_{\ast}\bigtriangleup_{I_{1};\kappa}^{\eta}f\ast
\Phi_{\ast}\bigtriangleup_{I_{2};\kappa}^{\eta}f\right]  ^{\wedge}\left(
\xi\right)  \right\vert =\left\vert \int e^{-iz\cdot\xi}F_{I_{1},I_{2}}\left(
z\right)  dz\right\vert \leq C_{N}\left(  \frac{1}{\left\vert \xi\right\vert
}\right)  ^{N}\int\left\vert \nabla^{N}F_{I_{1},I_{2}}\right\vert \left(
z\right)  dz\\
&  \leq C_{N}\left(  \frac{2^{s_{2}}}{\nu\left\vert \xi\right\vert }\right)
^{N}\left\vert \Phi\left(  I_{1}\right)  +\Phi\left(  I_{2}\right)
\right\vert \ell\left(  I_{1}\right)  ^{-\frac{2}{p}}\left\Vert f_{1}%
\right\Vert _{L^{p}}\ell\left(  I_{2}\right)  ^{-\frac{2}{p}}\left\Vert
f_{2}\right\Vert _{L^{p}}\\
&  \approx C_{N}\nu\left(  \frac{2^{s_{2}}}{\nu2^{r}}\right)  ^{N}%
2^{-2s_{1}-s_{2}}2^{\frac{2}{p}s_{1}}2^{\frac{2}{p}s_{2}}\left\Vert
f_{1}\right\Vert _{L^{p}}\left\Vert f_{2}\right\Vert _{L^{p}}\ ,
\end{align*}
since $\Phi\left(  I_{1}\right)  +\Phi\left(  I_{2}\right)  $ is roughly a
three dimensional rectangle of dimensions $2^{-s_{1}}\times2^{-s_{1}}%
\times\left(  \sin\nu\right)  2^{-s_{2}}$.

As a consequence we have%
\begin{align*}
&  \left(  \int_{A\left(  0,R\right)  }\left(  \left\vert \mathcal{E}%
\mathsf{Q}_{s_{1},U_{1}}^{\eta}f_{1}\left(  \xi\right)  \right\vert
\ \left\vert \mathcal{E}\mathsf{Q}_{s_{2},U_{2}}^{\eta}f_{2}\left(
\xi\right)  \right\vert \ \left\vert \mathcal{E}\mathsf{Q}_{s_{3},U_{3}}%
^{\eta}f_{3}\left(  \xi\right)  \right\vert \right)  ^{\frac{q}{3}}%
d\xi\right)  ^{\frac{3}{q}}\\
&  =\left(  \int_{A\left(  0,2^{r}\right)  }\left(  \left\vert \left[
\Phi_{\ast}\mathsf{Q}_{s_{1},U_{1}}^{\eta}f\ast\Phi_{\ast}\mathsf{Q}%
_{s_{2},U_{2}}^{\eta}f\right]  ^{\wedge}\left(  \xi\right)  \right\vert
\ \left\vert \widehat{\Phi_{\ast}\mathsf{Q}_{s_{3},U_{3}}^{\eta}f}\left(
\xi\right)  \right\vert \right)  ^{\frac{q}{3}}d\xi\right)  ^{\frac{3}{q}}\\
&  \leq\sum_{I_{1}\in\mathcal{G}_{s_{1}}\left[  U_{1}\right]  }\sum_{I_{2}%
\in\mathcal{G}_{s_{2}}\left[  U_{2}\right]  }\left(  \int_{A\left(
0,2^{r}\right)  }\left(  \left\vert \left[  \Phi_{\ast}\bigtriangleup
_{I_{1};\kappa}^{\eta}f\ast\Phi_{\ast}\bigtriangleup_{I_{2};\kappa}^{\eta
}f\right]  ^{\wedge}\left(  \xi\right)  \right\vert \ \left\vert \widehat
{\Phi_{\ast}\mathsf{Q}_{s_{3},U_{3}}^{\eta}f}\left(  \xi\right)  \right\vert
\right)  ^{\frac{q}{3}}d\xi\right)  ^{\frac{3}{q}}\\
&  \leq\sum_{I_{1}\in\mathcal{G}_{s_{1}}\left[  U_{1}\right]  }\sum_{I_{2}%
\in\mathcal{G}_{s_{2}}\left[  U_{2}\right]  }\left(  \int_{A\left(
0,2^{r}\right)  }\left(  C_{N}\nu\left(  \frac{2^{s_{2}}}{\nu2^{r}}\right)
^{N}2^{-2s_{1}-s_{2}}2^{\frac{2}{p}s_{1}}2^{\frac{2}{p}s_{2}}\left\Vert
f_{1}\right\Vert _{L^{p}}\left\Vert f_{2}\right\Vert _{L^{p}}\left\Vert
f_{3}\right\Vert _{L^{1}}\right)  ^{\frac{q}{3}}d\xi\right)  ^{\frac{3}{q}},
\end{align*}
which is approximately,%
\begin{align*}
&  C_{N}\nu^{1-N}2^{2s_{1}}2^{2s_{2}}\left(  \frac{2^{s_{2}}}{2^{r}}\right)
^{N}2^{-2s_{1}-s_{2}}2^{\frac{2}{p}s_{1}}2^{\frac{2}{p}s_{2}}\left\Vert
f_{1}\right\Vert _{L^{p}}\left\Vert f_{2}\right\Vert _{L^{p}}\left\Vert
f_{3}\right\Vert _{L^{1}}2^{r\frac{9}{q}}\\
&  \lesssim C_{N}\nu^{1-N}\left(  \frac{2^{s_{2}}}{2^{r}}\right)  ^{N}%
2^{s_{2}}2^{\frac{2}{p}s_{1}}2^{\frac{2}{p}s_{2}}\left\Vert f_{1}\right\Vert
_{L^{p}}\left\Vert f_{2}\right\Vert _{L^{p}}\left\Vert f_{3}\right\Vert
_{L^{p}}2^{r\frac{9}{q}}\\
&  \leq C_{N}\nu^{1-N}2^{\left(  \frac{9}{q}-N\right)  r}2^{\left(
N+1+\frac{4}{p}\right)  s_{2}}\left\Vert f_{1}\right\Vert _{L^{p}}\left\Vert
f_{2}\right\Vert _{L^{p}}\left\Vert f_{3}\right\Vert _{L^{p}}\\
&  <C_{N}\nu^{1-N}2^{\left(  \frac{9}{q}-N\right)  r}2^{\left(  N+1+\frac
{4}{p}\right)  \frac{r}{1+\delta}}\left\Vert f_{1}\right\Vert _{L^{p}%
}\left\Vert f_{2}\right\Vert _{L^{p}}\left\Vert f_{3}\right\Vert _{L^{p}}\ .
\end{align*}
If we choose $\frac{N+1+\frac{4}{p}}{1+\delta}+\frac{9}{q}-N<-1$, then
$2^{\left(  \frac{9}{q}-N\right)  r}2^{\left(  N+1+\frac{4}{p}\right)
\frac{r}{1+\delta}}<2^{-r}$ and so%
\[
\sum_{s_{1},s_{2}=1}^{\frac{r}{1+\delta}}\left(  \int_{A\left(  0,R\right)
}\left(  \left\vert \mathcal{E}\mathsf{Q}_{s_{1},U_{1}}^{\eta}f_{1}\left(
\xi\right)  \right\vert \ \left\vert \mathcal{E}\mathsf{Q}_{s_{2},U_{2}}%
^{\eta}f_{2}\left(  \xi\right)  \right\vert \ \left\vert \mathcal{E}%
\mathsf{Q}_{s_{3},U_{3}}^{\eta}f_{3}\left(  \xi\right)  \right\vert \right)
^{\frac{q}{3}}d\xi\right)  ^{\frac{3}{q}}\lesssim C_{N}\nu^{1-N}r^{2}%
2^{-r}\left\Vert f_{1}\right\Vert _{L^{p}}\left\Vert f_{2}\right\Vert _{L^{p}%
}\left\Vert f_{3}\right\Vert _{L^{p}},
\]
for all $s_{3}\leq\frac{r}{1-\delta}$ and for $N$ sufficiently large depending
only on $\delta$, $p$ and $q$, and in particular is independent of $s_{3}$.

Thus altogether we have%
\begin{align}
&  \left(  \int_{A\left(  0,2^{r}\right)  }\left(  \sum_{s_{1}\leq s_{2}%
\leq\frac{r}{1+\delta}\text{ and }s_{3}\leq\frac{r}{1-\delta}}\left\vert
\mathcal{E}\mathsf{Q}_{s_{1},U_{1}}^{\eta}f_{1}\left(  \xi\right)  \right\vert
\ \left\vert \mathcal{E}\mathsf{Q}_{s_{2},U_{2}}^{\eta}f_{2}\left(
\xi\right)  \right\vert \ \left\vert \mathcal{E}\mathsf{Q}_{s_{3},U_{3}}%
^{\eta}f_{3}\left(  \xi\right)  \right\vert \right)  ^{\frac{q}{3}}%
\ d\xi\right)  ^{\frac{3}{q}}\label{second half}\\
&  \leq\sum_{s_{1}\leq s_{2}\leq\frac{r}{1+\delta}\text{ and }s_{3}\leq
\frac{r}{1-\delta}}\left(  \int_{B\left(  0,2^{r}\right)  }\left(  \left\vert
\mathcal{E}\mathsf{Q}_{s_{1},U_{1}}^{\eta}f_{1}\left(  \xi\right)  \right\vert
\ \left\vert \mathcal{E}\mathsf{Q}_{s_{2},U_{2}}^{\eta}f_{2}\left(
\xi\right)  \right\vert \ \left\vert \mathcal{E}\mathsf{Q}_{s_{3},U_{3}}%
^{\eta}f_{3}\left(  \xi\right)  \right\vert \right)  ^{\frac{q}{3}}%
\ d\xi\right)  ^{\frac{3}{q}}\nonumber\\
&  \lesssim\left(  C_{N}\nu^{1-N}r^{2}2^{-r}\right)  \frac{r}{1-\delta
}\left\Vert f_{1}\right\Vert _{L^{p}}\left\Vert f_{2}\right\Vert _{L^{p}%
}\left\Vert f_{3}\right\Vert _{L^{p}}\lesssim C_{N}\nu^{1-N}r^{3}%
2^{-r}\left\Vert f_{1}\right\Vert _{L^{p}}\left\Vert f_{2}\right\Vert _{L^{p}%
}\left\Vert f_{3}\right\Vert _{L^{p}}.\nonumber
\end{align}

\end{proof}

\begin{remark}
Instead of appealing to Lemma \ref{scales} above, we could have fixed $\xi\in
A\left(  0,2^{r}\right)  $, and then chosen a patch $U_{k}$ with $1\leq
k\leq2$ such that the unit vector $\xi^{\prime}\equiv\frac{\xi}{\left\vert
\xi\right\vert }$ is at a positive angle depending on $\nu$, to the normal
vectors of $\Phi\left(  U_{k}\right)  $, and then integrated by parts along
the singular patch $\Phi\left(  U_{k}\right)  $. See e.g. \cite[Subsubsection
4.2.4]{Saw7} on tangential integration by parts for a similar calculation.
\end{remark}

\subsection{Wrapup of the proof}

\begin{proof}
[Proof continued]Now we can finish the proof. With $f_{k}\in L^{\infty}\left(
U_{k}\right)  \subset L^{p}\left(  U_{k}\right)  $ and $R=2^{r}$,
\begin{align*}
&  \left(  \int_{A\left(  0,R\right)  }\left(  \left\vert \mathcal{E}%
f_{1}\left(  \xi\right)  \right\vert \ \left\vert \mathcal{E}f_{2}\left(
\xi\right)  \right\vert \ \left\vert \mathcal{E}f_{3}\left(  \xi\right)
\right\vert \right)  ^{\frac{q}{3}}\ d\xi\right)  ^{\frac{3}{q}}\\
&  =\left(  \int_{A\left(  0,2^{r}\right)  }\left(  \left\vert \mathcal{E}%
\sum_{I_{1}\in\mathcal{G}\left[  U_{1}\right]  }\bigtriangleup_{I_{1};\kappa
}^{\eta}f_{1}\left(  \xi\right)  \right\vert \ \left\vert \mathcal{E}%
\sum_{I_{2}\in\mathcal{G}\left[  U_{2}\right]  }\bigtriangleup_{I_{2};\kappa
}^{\eta}f_{2}\left(  \xi\right)  \right\vert \ \left\vert \mathcal{E}%
\sum_{I_{3}\in\mathcal{G}\left[  U_{3}\right]  }\bigtriangleup_{I_{3};\kappa
}^{\eta}f_{3}\left(  \xi\right)  \right\vert \right)  ^{\frac{q}{3}}%
\ d\xi\right)  ^{\frac{3}{q}}\\
&  \lesssim\sum_{0\leq s_{1}\leq s_{2}\leq s_{3}}\left(  \int_{A\left(
0,2^{r}\right)  }\left(  \left\vert \mathcal{E}\mathsf{Q}_{s_{1},U_{1}}^{\eta
}f_{1}\left(  \xi\right)  \right\vert \ \left\vert \mathcal{E}\mathsf{Q}%
_{s_{2},U_{2}}^{\eta}f_{2}\left(  \xi\right)  \right\vert \ \left\vert
\mathcal{E}\mathsf{Q}_{s_{3},U_{3}}^{\eta}f_{3}\left(  \xi\right)  \right\vert
\right)  ^{\frac{q}{3}}\ d\xi\right)  ^{\frac{3}{q}},
\end{align*}
where the restriction to $I_{k}\in\mathcal{G}\left[  U_{k}\right]  $ can be
made by adapting the reduction argument in the proof of \cite[Lemma 2 on page
5]{Saw7}. Next we obtain from (\ref{first half}) and (\ref{second half}), that
the last line above is at most,%
\begin{align*}
&  \left\{  \sum_{s_{1}\leq s_{2}\leq s_{3}\text{ and }\frac{r}{1+\delta}\leq
s_{2},s_{3}\leq\frac{r}{1-\delta}}+\sum_{s_{1}\leq s_{2}\leq s_{3}\text{ and
}\frac{r}{1-\delta}\leq s_{3}}+\sum_{s_{1}\leq s_{2}\leq s_{3}\leq\frac
{r}{1-\delta}\text{ and }s_{1}\leq s_{2}\leq\frac{r}{1+\delta}}\right\} \\
&  \ \ \ \ \ \ \ \ \ \ \ \ \ \ \ \times\left(  \int_{A\left(  0,2^{r}\right)
}\left(  \left\vert \mathcal{E}\mathsf{Q}_{s_{1},U_{1}}^{\eta}f_{1}\left(
\xi\right)  \right\vert \ \left\vert \mathcal{E}\mathsf{Q}_{s_{2},U_{2}}%
^{\eta}f_{2}\left(  \xi\right)  \right\vert \ \left\vert \mathcal{E}%
\mathsf{Q}_{s_{3},U_{3}}^{\eta}f_{3}\left(  \xi\right)  \right\vert \right)
^{\frac{q}{3}}\ d\xi\right)  ^{\frac{3}{q}}\\
&  \lesssim\sum_{s_{1}\leq s_{2}\leq s_{3}\text{ and }\frac{r}{1+\delta}\leq
s_{2},s_{3}\leq\frac{r}{1-\delta}}\left(  \int_{A\left(  0,2^{r}\right)
}\left(  \left\vert \mathcal{E}\mathsf{Q}_{s_{1},U_{1}}^{\eta}f_{1}\left(
\xi\right)  \right\vert \ \left\vert \mathcal{E}\mathsf{Q}_{s_{2},U_{2}}%
^{\eta}f_{2}\left(  \xi\right)  \right\vert \ \left\vert \mathcal{E}%
\mathsf{Q}_{s_{3},U_{3}}^{\eta}f_{3}\left(  \xi\right)  \right\vert \right)
^{\frac{q}{3}}\ d\xi\right)  ^{\frac{3}{q}}\\
&  \ \ \ \ \ \ \ \ \ \ \ \ \ \ \ \ \ \ \ \ \ \ \ \ \ +C_{\kappa,\nu}\left\Vert
f_{1}\right\Vert _{L^{p}}\left\Vert f_{2}\right\Vert _{L^{p}}\left\Vert
f_{3}\right\Vert _{L^{p}}.
\end{align*}
Recall that $0<\delta<1$ and $\kappa>\frac{20}{\delta}$ are given. Now by
assumption, there is $\nu>0$ depending only on $q>3$ such that (\ref{annular})
holds with $2\delta$ and $\infty$ in place of $\delta$ and $p$ respectively,
and this then gives%
\[
\left(  \int_{A\left(  0,R\right)  }\left(  \left\vert \mathcal{E}f_{1}\left(
\xi\right)  \right\vert \ \left\vert \mathcal{E}f_{2}\left(  \xi\right)
\right\vert \ \left\vert \mathcal{E}f_{3}\left(  \xi\right)  \right\vert
\right)  ^{\frac{q}{3}}\ d\xi\right)  ^{\frac{3}{q}}\lesssim R^{\varepsilon
}\left\Vert f_{1}\right\Vert _{L^{\infty}}\left\Vert f_{2}\right\Vert
_{L^{\infty}}\left\Vert f_{3}\right\Vert _{L^{\infty}}.
\]

Finally, with $R=2^{r}$, we use that $B\left(  0,R\right)  $ is a union of the
unit ball and at most $r$ annuli of the form $A\left(  0,2^{t}\right)
=B\left(  0,2^{t}\right)  \setminus B\left(  0,2^{t-1}\right)  $, in order to
obtain,%
\begin{align*}
&  \int_{B\left(  0,R\right)  }\left(  \left\vert \mathcal{E}f_{1}\left(
\xi\right)  \right\vert \ \left\vert \mathcal{E}f_{2}\left(  \xi\right)
\right\vert \ \left\vert \mathcal{E}f_{3}\left(  \xi\right)  \right\vert
\right)  ^{\frac{q}{3}}d\xi\\
&  =\sum_{t=1}^{r}\int_{A\left(  0,2^{t}\right)  }\left(  \left\vert
\mathcal{E}f_{1}\left(  \xi\right)  \right\vert \ \left\vert \mathcal{E}%
f_{2}\left(  \xi\right)  \right\vert \ \left\vert \mathcal{E}f_{3}\left(
\xi\right)  \right\vert \right)  ^{\frac{q}{3}}d\xi+\int_{B\left(  0,1\right)
}\left(  \left\vert \mathcal{E}f_{1}\left(  \xi\right)  \right\vert
\ \left\vert \mathcal{E}f_{2}\left(  \xi\right)  \right\vert \ \left\vert
\mathcal{E}f_{3}\left(  \xi\right)  \right\vert \right)  ^{\frac{q}{3}}d\xi\\
&  \lesssim\sum_{t=1}^{r}\left(  2^{t\varepsilon}\left\Vert f_{1}\right\Vert
_{L^{\infty}}\left\Vert f_{2}\right\Vert _{L^{\infty}}\left\Vert
f_{3}\right\Vert _{L^{\infty}}\right)  ^{\frac{q}{3}}+C_{\nu}^{\frac{q}{3}%
}\left(  \left\Vert f_{1}\right\Vert _{L^{\infty}}\left\Vert f_{2}\right\Vert
_{L^{\infty}}\left\Vert f_{3}\right\Vert _{L^{\infty}}\right)  ^{\frac{q}{3}%
}\\
&  \lesssim C_{\nu}^{\frac{q}{3}}R^{\varepsilon\frac{q}{3}}\left(  \left\Vert
f_{1}\right\Vert _{L^{\infty}}\left\Vert f_{2}\right\Vert _{L^{\infty}%
}\left\Vert f_{3}\right\Vert _{L^{\infty}}\right)  ^{\frac{q}{3}%
},\ \ \ \ \ \text{for all }R\geq1,
\end{align*}
which is $\mathcal{E}_{\mathop{\rm disj}\nu}\left(  \otimes_{3}L^{\infty
}\rightarrow L^{\frac{q}{3}};\varepsilon\right)  $, where we have used
$R^{\varepsilon}+C_{\nu}\leq C_{\nu}R^{\varepsilon}$ for $R^{\varepsilon
},C_{\nu}\geq1$. By Theorem \ref{main}, this is equivalent to the Fourier
extension conjecture, and this completes the proof of Theorem
\ref{main Alpert}.
\end{proof}

\section{Appendix: Application to the probabilistic Fourier extension theorem}

Here we briefly describe an alternate, and arguably simpler, approach to
proving the three dimensional \emph{probabilistic} Fourier extension
inequality obtained in \cite{Saw7}. However, this alternate approach has
little likelihood of being extended to the Knapp segment at the boundary of
allowable exponents, something that is not out of the question for the proof
given in \cite{Saw7}.

The square function formulation of the probabilistic inequality proved in
\cite{Saw7} is,%
\begin{equation}
\left\Vert \mathcal{S}_{\mathop{\rm Fourier}}f\right\Vert _{L^{q}\left(
\lambda_{n}\right)  }\lesssim\left\Vert f\right\Vert _{L^{q}\left(  B\left(
0,\frac{1}{2}\right)  \right)  }\ ,\label{prob ext square}%
\end{equation}
where $\mathcal{S}_{\mathop{\rm Fourier}}$ is the \emph{Fourier} square
function defined by%
\begin{equation}
\mathcal{S}_{\mathop{\rm Fourier}}f\equiv\left(  \sum_{I\in\mathcal{G}\left[
U\right]  }\left\vert \mathcal{E}\mathbf{1}_{U_{0}}\bigtriangleup_{I;\kappa
}^{n-1,\eta}f\right\vert ^{2}\right)  ^{\frac{1}{2}}.\label{def T square}%
\end{equation}
Here the extension operator $\mathcal{E}$ is defined by%
\begin{equation}
\mathcal{E}f\left(  \xi\right)  \equiv\int_{B_{2}\left(  0,\frac{1}{2}\right)
}e^{-i\Phi\left(  x\right)  \cdot\xi}f\left(  x\right)  dx,\ \ \ \ \ \xi
\in\mathbb{R}^{n},\label{def T}%
\end{equation}
for $f\in L^{p}\left(  B_{2}\left(  0,\frac{1}{2}\right)  \right)  $. Thus
$\mathcal{E}f=\mathcal{F}\Phi_{\ast}\left(  f\lambda_{2}\right)
=\widehat{\Phi_{\ast}\left(  f\lambda_{2}\right)  }$, where $\Phi_{\ast}\nu$
denotes the pushforward of a measure $\nu$ under the map $\Phi$, and
$\lambda_{2}$ is Lebesgue measure in the plane.

Here we \emph{conjecture} the square function analogue of the second part of
Theorem \ref{main Alpert}.

\begin{definition}
Suppose $n=3$. Let $\varepsilon,\nu>0$, $0<\delta<1$, and $1<q<\infty$. Denote
by $\mathcal{A}_{\mathop{\rm disj}\nu}^{\kappa,\delta,\mathop{\rm
square}}\left(  \otimes_{3}L^{q}\rightarrow L^{\frac{q}{3}};\varepsilon
\right)  $ the disjoint \emph{smooth Alpert square function }trilinear
inequality%
\[
\left\Vert \mathcal{S}_{\mathop{\rm Fourier}}\mathsf{Q}_{U_{1}}^{s_{1}}%
f_{1}\ \mathcal{S}_{\mathop{\rm Fourier}}\mathsf{Q}_{U_{2}}^{s_{2}}%
f_{2}\ \mathcal{S}_{\mathop{\rm Fourier}}\mathsf{Q}_{U_{3}}^{s_{3}}%
f_{3}\right\Vert _{L^{\frac{q}{3}}\left(  A\left(  0,2^{s}\right)  \right)
}\leq C_{\varepsilon,\nu}2^{\varepsilon s}\left\Vert f_{1}\right\Vert
_{L^{\infty}\left(  U\right)  }\left\Vert f_{2}\right\Vert _{L^{\infty}\left(
U\right)  }\left\Vert f_{3}\right\Vert _{L^{\infty}\left(  U\right)  }\ ,
\]
taken over all $s_{1}\leq s_{2}$ and $s_{2},s_{3}\in\left(  \left(
1-\delta\right)  s,\left(  1+\delta\right)  s\right)  $, all $f_{k}\in
L^{q}\left(  U_{k}\right)  $, and all triples $\left(  U_{1},U_{2}%
,U_{3}\right)  \subset U^{3}$ that satisfy the weak $\nu$-disjoint condition,%
\[
\mathop{\rm diam}\left[  \Phi\left(  U_{k}\right)  \right]  \approx
\mathop{\rm dist}\left[  \Phi\left(  U_{k}\right)  ,\bigcup_{j:\ j\neq k}%
\Phi\left(  U_{j}\right)  \right]  \gtrsim\nu,\text{ for }1\leq k\leq3,
\]

\end{definition}

\begin{conjecture}
\label{alternate}Let $0<\delta<1$ and $\kappa>\frac{10}{\delta}$. The square
function inequality (\ref{prob ext square}) holds in $\mathbb{R}^{3}$ \emph{if
and only if} for every $q>3$ there is $\nu>0$ such that the disjoint
\emph{smooth Alpert square function }trilinear inequality $\mathcal{A}%
_{\mathop{\rm disj}\nu}^{\kappa,\delta,\mathop{\rm square}}\left(  \otimes
_{3}L^{\infty}\rightarrow L^{\frac{q}{3}};\varepsilon\right)  $ holds for all
$\varepsilon>0$.
\end{conjecture}

It should be possible to prove this square function variant by tracing through
the proof of Theorem \ref{main Alpert} above and making modifications for the
square function. See for example \cite{RiSa2}, where these modifications are
carried out for `father' wavelets $\varphi_{I}$ smoothly adapted to a square
$I$, and their projections $\bigtriangleup_{I}f\equiv\left\langle
f,\varphi_{I}\right\rangle \varphi_{I}$. However, an additional difficulty
arises for smooth Alpert wavelets $h_{I;\kappa}^{\eta}$, since the
pseudoprojections $\bigtriangleup_{I;\kappa}^{\eta}f\equiv\left\langle
S_{\kappa,\eta}^{-1}f,h_{I;\kappa}\varphi_{I}\right\rangle h_{I;\kappa}^{\eta
}$ involve an operator $S_{\kappa,\eta}^{-1}$ that is bounded only on $L^{p}$
for $1<p<\infty$, and \emph{not} on $L^{\infty}$. In this case, parabolic
rescalings must be performed with $f\in L^{p}\left(  U\right)  $, which
introduces an additional growth factor $C_{p}2^{\frac{2}{p}s}$, that turns out
to be harmless because one can take $p$ arbitrarily large. As mentioned
earlier, the proof of this conjecture will be addressed in a future paper.

Conjecture \ref{alternate} can be used to give an alternate proof of the
probabilistic analogue of Fourier extension in \cite{Saw7}. In fact, the case
$n=3$ of Proposition 34 in \cite{Saw7} says that for $q>3$, there is
$\varepsilon_{q}>0$ such that for every $s\in\mathbb{N}$, and every $f\in
L^{q}\left(  U\right)  $, we have,
\begin{equation}
\mathbb{E}_{\pm}\left\Vert \mathcal{E}\left[  \left(  \pm\mathsf{Q}_{U}%
^{s}\right)  ^{\spadesuit}f\right]  \right\Vert _{L^{q}\left(  B\left(
0,2^{s}\right)  \right)  }\lesssim2^{-s\varepsilon_{q}}\left\Vert f\right\Vert
_{L^{q}\left(  U\mathbb{R}^{n-1}\right)  },\label{eps}%
\end{equation}
where $\mathbb{E}_{\pm}$ denotes the average over the `martingale transforms'
$\left(  \pm\mathsf{Q}_{U}^{s}\right)  ^{\spadesuit}f$, and the implied
constant depends on $q$ and $U$, but is independent of $s\in\mathbb{N}$.
Indeed, (\ref{eps}) is easily obtained by computing the norms when $p=2$ and
$4$, and then using the expectation $\mathbb{E}_{\pm}$ to eliminate a large
number of off diagonal terms in the $L^{4}$ estimate, resulting in a geometric
decay in $s$ that exactly balances the growth in $s$ for the $L^{2} $ estimate
when $q=3$. Moreover, we can enlarge the ball $B\left(  0,2^{s}\right)  $ to
$B\left(  0,2^{\left(  1+\delta\right)  s}\right)  $ provided $\delta\leq
C\varepsilon$. See \cite[Section 5]{Saw7} for details.

Then using Khintchine's inequality, we obtain the square function formulation
of (\ref{eps}), namely that for $q>3$, there is $\varepsilon_{q}>0$ such that
for every $s\in\mathbb{N}$, and every $f\in L^{q}\left(  U\right)  $, we have,%
\[
\left\Vert \mathcal{S}_{\mathop{\rm Fourier}}\mathsf{Q}_{U}^{s}f\right\Vert
_{L^{q}\left(  B\left(  0,2^{\left(  1+\delta\right)  s}\right)  \right)
}\lesssim2^{-s\varepsilon_{q}}\left\Vert f\right\Vert _{L^{p}\left(  U\right)
},\ \ \ \ \ s\in\mathbb{N}.
\]
Thus we conclude that $\mathcal{A}_{\mathop{\rm disj}\nu}^{\kappa
,\delta,\mathop{\rm square}}\left(  \otimes_{3}L^{\infty}\rightarrow
L^{\frac{q}{3}};\varepsilon\right)  $ holds, even with a negative exponent
$\varepsilon$. Now Conjecture \ref{alternate} completes the alternate proof of
the square function formulation (\ref{prob ext square}) of the probabilistic
Fourier extension theorem in the case $n=3$.


\begin{thebibliography}{999999}                                                                                           %
\bibitem[BeCaTa]{BeCaTa}\textsc{J. Bennett, A. Carbery and T. Tao}, \textit{On
the multilinear restriction and Kakeya conjectures}, Acta. Math. \textbf{196}
(2006), 261-302.

\bibitem[BoGu]{BoGu}\textsc{J. Bourgain and L. Guth},\textit{\ Bounds on
oscillatory integral operators based on multilinear estimates}, Geom. Funct.
Anal. \textbf{21} (6), 2011, 1239--1295.

\bibitem[Bus]{Bus}\textsc{S. Buschenhenke,} \textit{Factorization in Fourier
restriction theory and near extremizers}, Math. Nachr. \textbf{297} (1), 2024, 195--208.

\bibitem[Car]{Car}\textsc{A. Carbery,} \textit{Restriction implies
B\^{o}chner-Riesz for paraboloids}, Math. Proc. Cambridge Philo. Soc.
\textbf{111 }no. 3, (1992), 525-529.

\bibitem[CaSj]{CaSj}\textsc{L. Carleson and P. Sj\"{o}lin,}
\textit{Oscillatory integrals and a multiplier problem for the disc}, Studia
Math. \textbf{44} (1972), 287--299.

\bibitem[Fef]{Fef}\textsc{C. Fefferman, }\textit{A note on spherical summation
multipliers}, Israel J. Math. (1973) \textbf{15}, 44--52.

\bibitem[MuOl]{MuOl}\textsc{C. Muscaru and I. Oliveira,} \textit{A new
approach to the Fourier extension problem for the paraboloid}, Analysis and
PDE \textbf{17} (2024), No. 8, 2841-2921.

\bibitem[RiSa2]{RiSa2}\textsc{C. Rios and E. Sawyer,} \textit{Equivalence of
linear and trilinear Kakeya conjectures in three dimensions},\textit{\ }%
\texttt{arXiv:2507.21315}.

\bibitem[Saw7]{Saw7}\textsc{E. Sawyer,} \textit{A probabilistic analogue of
the Fourier extension conjecture, }\texttt{arXiv:2311.03145v14}.

\bibitem[Ste]{Ste}\textsc{E. M. Stein,} \textit{Some problems in harmonic
analysis}, Harmonic analysis in Euclidean spaces (Proc. Sympos. Pure Math.,
Williams Coll., Williamstown, Mass., 1978), Part 1, pp. 3-20, Proc. Sympos.
Pure Math., \textbf{XXXV}, Part, Amer. Math. Soc., Providence, R.I., 1979.

\bibitem[Ste2]{Ste2}\textsc{E. M. Stein,} \textit{Harmonic Analysis:
real-variable methods, orthogonality, and oscillatory integrals}%
,\textit{\ }Princeton University Press, Princeton, N. J., 1993.

\bibitem[Tao]{Tao}\textsc{T. Tao}, \textit{The B\^{o}chner-Riesz conjecture
implies the restriction conjecture}, Duke Math. J. \textbf{96} (1999), no. 2, 363-375.

\bibitem[Tao2]{Tao2}\textsc{T. Tao}, \textit{Sharp bounds for multilinear
curved Kakeya, restriction and oscillatory integral estimates away from the
endpoint}, Mathematika \textbf{66} (2020), 517-576.

\bibitem[TaVaVe]{TaVaVe}\textsc{T. Tao, A. Vargas, and L. Vega}, \textit{A
bilinear approach to the restriction and Kakeya conjectures}, JAMS \textbf{11}
(1998), no. 4, 967--1000.

\bibitem[WaZa]{WaZa}\textsc{Hong Wang and Joshua Zahl,} \textit{Volume
estimates for unions of convex sets, and the Kakeya set conjecture in three
dimensions}, \texttt{arXiv:2502.17655v1.}

\bibitem[Zyg]{Zyg}\textsc{A. Zygmund,} \textit{On Fourier coefficients and
transforms of functions of two variables}, Studia Mathematica \textbf{50}
(1974), no. 2, 189-201.
\end{thebibliography}
\end{document}